\documentclass[12pt]{amsart}
\usepackage{amsmath,amssymb,bm,enumerate,graphicx,multido,multirow,pst-3dplot,pstricks,pst-node,caption,subcaption}
\usepackage[mathscr]{euscript}

\newtheorem{theorem}{Theorem}[section]
\newtheorem{hypothesis}[theorem]{Hypothesis}
\newtheorem{definition}[theorem]{Definition}
\newtheorem{lemma}[theorem]{Lemma}
\newtheorem{proposition}[theorem]{Proposition}
\newtheorem{corollary}[theorem]{Corollary}

\theoremstyle{definition}
\newtheorem{example}[theorem]{Example}

\newtheorem{question}{Question}

\def\D{{\mathscr{D}}}
\def\PP{{\mathscr{P}}}

\def\B{{\mathscr{B}}}
\def\C{{\mathscr{C}}}
\def\I{\mathscr{I}}

\def\Sym{{\mathrm{Sym}}}

\def\A{\mathscr{A}}
\def\CC{\bm{\mathscr{C}}}
\def\Bdy{\partial}
\def\c{{\rm c}}

\newcommand{\ff[1]}{F^{(#1)}}

\newcommand{\Sq[1]}{(#1)\mu_B}
\newcommand{\cc[1]}{\CC^{(#1)}}

\definecolor{amber}{rgb}{1.0, 0.7, 0.0}

\usepackage{anysize}
\marginsize{2cm}{2cm}{2cm}{2cm}

\title{Block-transitive designs with a poset of imprimitive partitions}
\author{Carmen Amarra$^{\rm\lowercase{b}}$}
\author{Alice Devillers$^{\rm\lowercase{a,\ast}}$ }
\author{Cheryl E. Praeger$^{\rm \lowercase{a}}$}
\address{$^{\rm\lowercase{a}}$Centre for the Mathematics of Symmetry and Computation, The University of Western Australia, 35 Stirling Highway, Crawley, Western Australia 6009, Australia}
\email{alice.devillers@uwa.edu.au, cheryl.praeger@uwa.edu.au}
\address{$^{\rm\lowercase{b}}$ Institute of Mathematics, University of the Philippines Diliman, C.P. Garcia Avenue, Quezon City 1101, Philippines}
\email{mcamarra@math.upd.edu.ph}
\address{$^{\ast}$Corresponding author}

\date{\today}

\begin{document}

\begin{abstract}

    We study block designs which admit an automorphism group that is transitive on blocks and points, and leaves invariant every partition in a given finite poset of partitions of the point set. The full stabiliser $G$  of all the partitions in the poset is a generalised wreath product. We use the theory of generalised wreath products to give necessary and sufficient conditions, in terms of  the `array' of a point-subset $B$, for the set of $G$-images of $B$ to form the block-set of a $G$-block-transitive $2$-design. This generalises  previous results for the special cases where the poset is a chain or an anti-chain. We also give explicit infinite families of examples of $2$-designs for each poset involving three proper partitions, and for the famous $N$-poset with four partitions.  (Posets with two proper partitions have been treated previously.) This suggests the problem of finding explicit examples for other posets.

\noindent{This work forms part of the Australian Research Council Discovery Grant project  
DP200100080. }
    
\end{abstract}

\maketitle

\section{Introduction} \label{s:intro}

In a designed statistical experiment, the experimental units are sometimes partitioned according to different factors that may influence the observations, such that in the case of two or more factors the corresponding set of partitions is a partially ordered set under refinement. A class of such posets was described by Nelder in \cite{N}, and other examples by Throckmorton in \cite{Thr}.
We study block designs which admit an automorphism group $G$ that is transitive on blocks and preserves a given poset $(\CC^*,\preccurlyeq)$ of partitions of the point set, that is, $G$ leaves invariant every partition $\C \in \CC^*$.

\begin{definition}\label{d:Iimp}
Given a partially ordered set $\I = (I,\preccurlyeq)$ and a nonempty set $\PP$. Let $(\CC^*,\preccurlyeq')$ be a poset of partitions of $\PP$ such that $(\CC^*,\preccurlyeq') \cong (I,\preccurlyeq)$.  A group $G\leqslant \Sym(\PP)$ is said to be \emph{$\I$-imprimitive} on $\PP$ if $G$ is transitive on $\PP$ and $G$ leaves invariant every partition in $\CC^*$.
\end{definition}

\noindent
 Since the relations $\preccurlyeq$ and $\preccurlyeq'$ are often very closely linked, we usually use the same symbol $\preccurlyeq$ for both. If $|I|=1$,  then the poset $(\CC^*,\preccurlyeq)$ consists of a single trivial partition with all parts of size $1$, and if this partition, and the partition with a single part, are the only $G$-invariant partitions, then $G$ is primitive on $\PP$. For this reason we are interested in posets $(I,\preccurlyeq)$ with $|I|\geq 2$. In our analysis we work with a partition poset $(\CC,\preccurlyeq)$ which may contain $(\CC^*,\preccurlyeq)$ as a proper subset. We use the notation $\CC^*$ in Definition~\ref{d:Iimp} to be consistent with later usage, see Subsection~\ref{s:posetbkstr} for more details.

A \emph{block design} is an incidence structure $\D = (\PP,\B)$ consisting of a finite set $\PP$ of \emph{points} and a collection of point-subsets called \emph{blocks}, such that, for some constants $k,r > 0$, every block has size $k$ and every point appears in exactly $r$ blocks. If the point set $\PP$ has size $v$, and if, in addition, for some constants $t,\lambda > 0$ any $t$-subset of points appears in $\lambda$ blocks, then the design is a $t$-$(v,k,\lambda)$ design. An \emph{automorphism} of $\D$ is a permutation of the set $\PP$ that preserves the block set $\B$, and the design $\D$ is \emph{point-transitive} (respectively, \emph{block-transitive}) if it admits a subgroup of automorphisms that acts transitively on $\PP$ (respectively, $\B$). It was shown in 1967 by Block in \cite{Block67} that, for $t \geq 2$, block-transitivity implies point-transitivity.
It was proved by Delandtsheer and Doyen \cite{DD} that nearly all block-transitive $2$-designs are point-primitive, and in \cite{CP93} general constructions of families of point-imprimitive $2$-designs were described by Cameron and the third author. The designs constructed involved one or two nontrivial invariant partitions.

Our goal in this paper is a vast generalisation of the approach of \cite{CP93}, developing a general theory of block-transitive $\I$-imprimitive $2$-designs for arbitrary  finite posets $\I=(I,\preccurlyeq)$ with $|I|\geq 2$. We also give explicit constructions of such $2$-designs for small posets. The designs that we study admit a \emph{generalised wreath product} as a group of automorphisms. The generalised wreath products considered are those introduced by Bailey, Rowley, Speed and the third author in \cite{BPRS} as permutation groups, based on a partially ordered set $(I,\preccurlyeq)$, see Subsection~\ref{s:posets}. They were constructed independently by Holland~\cite{Ho} and Silcock~\cite{Sil} in a more restricted setting. These group actions arise as a specialisation of a construction which Wells \cite[\S 7]{W} described in 1976 for actions of semigroups, and which he called the `wreath product of an ordered set of actions'. It was shown in \cite[Theorem A]{BPRS} that Wells' wreath product action is a group action if the input actions are all group actions and the poset $(I,\preccurlyeq)$ satisfies the maximal condition.  Moreover, for a certain important sub-family of posets $(I,\preccurlyeq)$, the generalised wreath products we consider are the automorphism groups of structures described by J. A. Nelder~\cite{N}, and frequently used to analyse designed statistical experiments.
The general question we address is the following.

\begin{question}\label{qu}
    Given a partially ordered set $\I$, does there exist a $2$-design admitting a block-transitive group of automorphisms that is $\I$-imprimitive on points? If so, give an explicit construction. 
\end{question}

Partial answers are available in the literature when the poset is a chain or anti-chain. We explain these concepts and give a brief overview of these results in Subsection~\ref{sec:chains}. We also mention in Subsection~\ref{sec:chains} some examples of  $\I$-imprimitive block-transitive $2$-designs which have `accidentally' occurred in various classification results in the literature.

 In this paper we use results from \cite{BPRS} about $\I$-imprimitive groups such that $\I=(I\preccurlyeq)$ satisfies a mild finiteness condition called  the \emph{maximal condition}, that is to say,  every nonempty subset $J\subseteq I$ contains a \emph{maximal element} (an element $j\in J$  such that there is no element $i\in J\setminus\{j\}$ with $j\preccurlyeq i$). While every finite poset $\I$ satisfies  this condition, some of our results hold for all (not necessarily finite) $\I$ satisfying the maximal condition, namely those in Sections~\ref{s:dirpdt} and~\ref{s:kron}, so in Section~\ref{s:prelims} we refrain from assuming finiteness until it is needed.

In analogy with the construction in \cite{CP93}, for a poset $\I=(I\preccurlyeq)$ satisfying the maximal condition, we define the point set as a Cartesian product $\PP = \prod_{i \in I} \Delta_i$ of sets $\Delta_i$, with $|\Delta_i|=e_i \geq 2$ for each $i\in I$. For any subset $J \subseteq I$ we define $\Delta_J := \prod_{i \in J} \Delta_i$, with the convention that $\Delta_\varnothing$ is a singleton set, and we let $\pi_J$ denote the natural projection $\PP \rightarrow \Delta_J$. A subset $J\subseteq I$ is called \emph{ancestral} if $j\in J$ and $j\preccurlyeq i$ implies that $i\in J$. In particular, $A[i] := \{ j \in I \ | \ i\preccurlyeq j\}$ is ancestral for each $i\in I$.
We denote by $\A(\I)$  the family of ancestral subsets of $I$. For each $J \in \A(\I)$ the sets $\{ \bm{\varepsilon} \in \PP \ | \ \bm{\varepsilon} \pi_J = \bm{\nu} \}$, as $\bm{\nu}$ ranges over $\Delta_J$, form a partition $\C_J$ of $\PP$, and these partitions $\C_J$ form a poset $(\CC,\preccurlyeq)$ under the operation of refinement. In particular the  subset $\CC^* = \{ \C_{A[i]} \ | \ i \in I \}$ forms a sub-poset $(\CC^*,\preccurlyeq)$ of $(\CC,\preccurlyeq)$ isomorphic to $\I=(I,\preccurlyeq)$.

For a subset $B \subseteq \PP$, the \emph{array function} $\chi_B$ of $B$  is the function
    \begin{equation} \label{arrayfn}
    \chi_B \ : \ \bigcup_{J \in \A(\I)} \Delta_J \rightarrow \mathbb{N}_{\geq 0}, \ \text{where $\forall\,J \in \A(\I)$ and $\forall\,\bm{\nu} \in \Delta_J$,} \ (\bm{\nu})\chi_B := |\{ \bm{\varepsilon} \in B \ | \ \bm{\varepsilon} \pi_J = \bm{\nu} \}|
    \end{equation}
and the \emph{array} of $B$ is the multiset of all the image sets $(\bm{\nu})\chi_B$. As above we consider the $1$-design $\D = (\PP,\B)$ with $\B=B^F$ the set of all images of the `base block' $B\subseteq \PP$ under elements of a suitable $\I$-imprimitive subgroup $F\leq \Sym(\PP)$, namely $F$ is a generalised wreath product $\prod_{(I,\preccurlyeq)}(G_i,\Delta_i)$ with each $G_i\leq \Sym(\Delta_i)$ acting $2$-transitively on $\Delta_i$. We describe such groups in detail in Section~\ref{ss:genwr}. Our first major result gives necessary and sufficient arithmetic conditions, in terms of the array function $\chi_B$, for $\D = (\PP,\B)$ to be a $2$-design.

\begin{theorem} \label{t:gwp-design}
Let $\I=(I,\preccurlyeq)$ be a finite partially ordered set  with $|I|\geq 2$, and  for each $i \in I$ let $\Delta_i$ be a finite set of size $e_i \geq 2$, and set $\PP = \prod_{i \in I} \Delta_i$, of size $v = \prod_{i \in I} e_i$. For each $i\in I$ let $G_i \leq \Sym(\Delta_i)$ acting $2$-transitively on $\Delta_i$, and let $F = \prod_{(I,\preccurlyeq)} (G_i,\Delta_i)$, as in Section~\ref{ss:genwr}. Let $B\subseteq \PP$ of size $k$, and set $\B := B^F$. Then $\D = (\PP,\B)$ is a $2$-design if and only if, for each proper nonempty ancestral subset $J\subset I$, 
    \begin{equation} \label{2des}
       \sum_{S \subseteq \Bdy{J}}(-1)^{|S|} \left(\sum_{\bm{\nu} \in \Delta_{J \cup S}}  \big((\bm{\nu})\chi_B\big)^2\right) 
    = \frac{k(k-1)}{v-1} \left( \prod_{i \in \Bdy{J}} (e_i - 1) \right) \left( \prod_{j \in (J \cup \Bdy{J})^\c} e_j \right)
    \end{equation}
where $J^\c$ denotes the complement of $J$ in $I$, and $\Bdy{J}$ is the set of all maximal elements in $J^\c$.
\end{theorem}

To demonstrate how Theorem~\ref{t:gwp-design} may be used to construct new $\I$-imprimitive block-transitive $2$-designs, we construct infinite families of such designs for each poset $(I,\preccurlyeq)$ with $|I|=3$ that is not a chain or an antichain (see Section \ref{sec:chains} for these posets), and also for the famous `N'-poset  in Figure~\ref{fig:posets-4} which was originally described by Throckmorton (see \cite[Diagram(7), Figure 3, p. 26]{Thr}, or  \cite[Example 1, Figure 3]{BPRS}). As a result of our new constructions in Section~\ref{sec:ex} (and previous work on chains and antichains which we describe in Section~\ref{sec:chains}), we have the following theorem. 
\begin{theorem}
For each partially ordered set $\I = (I,\preccurlyeq)$ with $2\leq |I| \leq 3$, and for $\I$ the `N'-poset in Figure~\ref{fig:posets-4},  there exists infinitely many $2$-designs $\D = (\PP,\B)$ such that $\D$ admits a block-transitive automorphism group $G$ which is $\I$-imprimitive on $\PP$.
\end{theorem}

\subsection{Previous results on chains and antichains}\label{sec:chains}

If $\I=(I,\preccurlyeq)$ is an \emph{$s$-chain} with  $s \geq 3$, that is to say, $|I|=s$ and $\preccurlyeq$ is a total order, then the design is said to be \emph{$s$-chain-imprimitive}. For a given point set $\PP$ admitting an $s$-chain-imprimitive group $G$ of automorphisms, and block set $\B$ that forms a single orbit under $G$, we provided in  \cite{chainspaper} necessary and sufficient conditions, in terms of the array, in order for the block-transitive, $s$-chain-imprimitive incidence structure $(\PP,\B)$ to be a $2$-design. 
We note that  \cite[Theorem 1.3]{chainspaper} is a special case of Theorem \ref{t:gwp-design} for $(I,\preccurlyeq)$ being an $s$-chain (we explain this in more detail after the proof of Theorem \ref{t:gwp-design} in Section \ref{s:proof}).
Furthermore, for any $s\geq 3$, we gave an explicit construction for an infinite family of block-transitive, $s$-chain-imprimitive $2$-designs in \cite{chainspaper}, and an infinite family of flag-transitive, $s$-chain-imprimitive $2$-designs in \cite{ftchainspaper}.

\medskip 

The partially ordered set $\I = (I,\preccurlyeq)$ is an \emph{$s$-antichain}, for arbitrary $s \geq 2$, if $|I| = s$ and $i \preccurlyeq j$ if and only if $i = j$. We call such designs \emph{$s$-grid-imprimitive}, since the point set can be identified with a Cartesian product $\Delta_1 \times \ldots \times \Delta_s$, which can be visualised as an $s$-dimensional grid. The case where $\I$ is a $2$-antichain  corresponds to two point-partitions with the property that any two sets that belong in different partitions intersect in exactly one point, and the point set can be arranged in a rectangular grid such that the two nontrivial partitions are the set of rows and the set of columns. For brevity we refer to $2$-grid-imprimitive designs as grid-imprimitive. A systematic approach to constructing $G$-block-transitive, $G$-grid-imprimitive $2$-designs began as part of the study of block-transitive point-imprimitive designs \cite[Section3]{CP93}, and was developed further by Alavi, Daneshkhah, and the second and third authors in \cite{grids22}, and general constructions and explicit examples of block-transitive, grid-imprimitive 2-designs were described in \cite{grids22}. The approach in \cite{CP93} was developed combinatorially and computationally by Brasi\v{c} et. al. \cite{BMV2018, BMSVV} to classify flag-transitive grid-imprimitive designs with bounded parameters. A number of examples of grid-imprimitive, block-transitive $2$-designs arise from examples where the group $\widehat{G} = S_n \wr S_2$ acts point-primitively and its elements preserve or transpose a square $n \times n$ grid, as described in \cite{grids22} (see the design $\widehat{\mathcal{D}}(\Delta)$ in \cite[(1.3)]{grids22}) and \cite{BMV2018}. The index $2$ base group $S_n^2$ is grid-imprimitive; it acts block-transitively if and only if a block-stabiliser $\widehat{G}_B$ in $\widehat{G}$ also interchanges the rows and columns of the grid. Through this observation some $G$-block-transitive designs with $G$ grid-imprimitive have been found either by computational enumeration (for example, \cite[Examples 8.1 and 8.2]{BMV2018}), or as part of theoretical classifications, see examples below. 

In further work with Alavi and Daneshkhah \cite{multigrids} we generalised the analysis of \cite{grids22} to an $s$-antichain $\I=(I\preccurlyeq)$ for arbitrary $s \geq 2$. For a given point set $\PP$ admitting an $s$-grid-imprimitive group $G$ of automorphisms, and block set $\B$ that forms a single orbit under $G$, we were able to find necessary and sufficient conditions, in terms of a set of parameters called the \emph{array} (defined in \eqref{arrayfn}), in order for the block-transitive, $s$-grid-imprimitive incidence structure $(\PP,\B)$ to be a $2$-design. Note that each of the results \cite[Theorem 1.3]{multigrids} is a special case of Theorem \ref{t:gwp-design} for $(I,\preccurlyeq)$ an $s$-antichain (we explain this in more detail after the proof of Theorem \ref{t:gwp-design} in Section \ref{s:proof}). Furthermore we were able to construct an infinite family of block-transitive, $s$-grid-imprimitive $2$-designs with $s=3$, and one example with $s=4$. Additional examples with smaller block sizes can be found in other work:
	\begin{itemize}
    \item For $k = 4$, the classification by Zhan et. al. \cite{ZZC} of block-transitive $2$-$(v,4,\lambda)$ designs has produced two $G$-block-transitive, $G$-grid-imprimitive $2$-designs with $G = S_5 \times S_5$, $v = 25$, and $\lambda = 18$ or $72$ \cite[Lemmas 4.3 and 4.4]{ZZC}.
    \item  For $k = 5$, Zhao and Zhou~\cite{ZZ} found two 2-$(81,5,\lambda)$ designs  (with $\lambda=14112, 7056$ respectively) which  admit a block-transitive action of $S_9\times S_9$,  namely the designs with base block shown in \cite[Fig. 2, Fig. 3]{ZZ}, since we see from these figures that a block-stabiliser in $S_9\wr S_2$ interchanges the rows and columns of the grid.
	\item For $k = 6$, there is a flag-transitive biplane (that is, a symmetric $2$-$(v,k,2)$ design) that admits a block-transitive automorphism group that preserves a $3$-dimensional grid structure on points. Its parameters are $(v,k,\lambda) = (16,6,2)$ with automorphism group $2^4.S_6$, and there is a grid-imprimitive automorphism subgroup $2^4.S_3$ that preserves three independent point-partitions that form a $2 \times 2 \times 4$ grid (see \cite[Table 2, line 1]{multigrids}).
	\item For $k = 8$, we again refer to \cite[Table 3]{DP21}: There is a $2$-$(36,8,4)$ design with automorphism group $S_6$ which is flag-regular and preserves a $6 \times 6$ grid structure, and a $2$-$(15,8,4)$ design that has a block-transitive (but not flag-transitive) automorphism subgroup that preserves a $3 \times 5$ structure.
	\item For $k = 20$, there are two $2$-$(96,20,4)$ designs that admit a flag-transitive, grid-imprimitive automorphism group (see \cite{LPR} and \cite[Table 3]{DP21}). One of these designs has automorphism group $2^8.S_6$ with a subgroup $2^4.S_5$ that is also flag-transitive and preserves a $6 \times 16$ grid. The other design has automorphism group $2^6.S_5$ with a subgroup $2^4.S_5$ that is flag-transitive and preserves a $6 \times 16$ grid (in fact there are one invariant partition with parts of size $16$ and two independent partitions with parts of size $6$.)
	\end{itemize}

\section{Poset block structures and generalised wreath products of groups}\label{s:prelims}

We present here some background required to prove Theorem~\ref{t:gwp-design}, including important concepts about posets, and their associated block structures and  symmetry groups. Most of this information is taken from \cite{BPRS}.

\subsection{Posets and their ancestral subsets}\label{s:posets}

Let $\I=(I,\preccurlyeq)$ be a partially ordered set, that is, a set $I$ with a reflexive and transitive relation $\preccurlyeq$ , where we write $i\preccurlyeq j$ if $(i,j)$ lies in this relation, and we write $i\prec j$ if in addition $i\ne j$. 
As usual, we visualise a poset $\I=(I,\preccurlyeq)$ with nodes representing elements of $I$ and distinct nodes $i$ and $j$ joined by an edge with $i$ below $j$ if $i\prec j$.
A subset $J\subseteq I$ is called an \emph{ancestral} subset if, for each $j\in J$, the relation $j\prec i$ implies that $i\in J$. For example,  the set $I$ is clearly ancestral, and the empty set $\varnothing$ is ancestral since it vacuously satisfies this condition. Also, for each $i\in I$, setting 
\[
A(i) := \{ j \in I \ | \ i\prec j \}\quad \mbox{and}\quad A[i] := A(i) \cup \{i\} = \{ j \in I \ | \ i\preccurlyeq j \},
\]
it follows from the definitions that both $A(i)$ and $A[i]$ are ancestral. Observe that for any $i, j \in I$, we have $A[i] \supsetneq A[j]$ if and only if $i\prec j$.

For a subset $J\subseteq I$, an element $j\in J$ is said to be a \emph{maximal element} of $J$ if no element $i\in J$ is such that $j\prec i$, in other words  $A(j)\cap J=\varnothing$, and the poset $\I$ is said to \emph{satisfy the maximal condition} if every subset of $I$ contains at least one maximal element.

We denote by $\A(\I)$ the set of all ancestral subsets of $I$. It is straightforward to see that $\A(\I)$ is closed under taking set-theoretic union. 
 For any ancestral subset $J \in \A(\I)$, the \emph{border} $\Bdy{J}$ of $J$ is the set of all maximal elements in the complement $J^\c = I\setminus J$. Note that, for any $S \subseteq \Bdy{J}$ the union $J \cup S \in \A(\I)$. In particular, for any $i \in I$ we have $i \in \Bdy{A(i)}$. 
 Moreover, each $J \in \A(\I)$ can be written as $J = \bigcup_{j \in J} A[j]$.

\subsection{Poset block structures and partitions}\label{s:posetbkstr}

Let $\I=(I,\preccurlyeq)$ be a partially ordered set. For each $i \in I$ let $\Delta_i$ be a set of size $e_i \geq 2$, and for each subset $J \subseteq I$, write  $\Delta_J := \prod_{j \in J} \Delta_j$ (the Cartesian product). In particular we set $\PP:= \Delta_I= \prod_{i \in I} \Delta_i$, and we use the convention that $\Delta_{\varnothing}$ is a singleton. For our main result Theorem~\ref{t:gwp-design} we will assume that both $\I$ and the $e_i$ are finite. 

Now choose $J \subseteq I$, and recall the projection $\pi_J:\PP \to \Delta_J$ in \eqref{piJ}.  Let $\sim_J$ be the equivalence relation on $\PP$ such that
    \[ 
\mbox{for all $\bm{\delta}, \bm{\varepsilon} \in \PP$,}\quad    \bm{\delta} \sim_J \bm{\varepsilon} \quad \text{if and only if} \quad \bm{\delta} \pi_J = \bm{\varepsilon} \pi_J. 
    \]
Also let $\C_J$ denote the set of equivalence classes of $\sim_J$, so that
    \begin{equation} \label{partition} 
    \mbox{setting}\quad C_{\bm{\nu}} = \left\{ \bm{\delta} \in \PP \ \vline \ \bm{\delta} \pi_J = \bm{\nu} \right\} \quad 
     \mbox{for each $\bm{\nu} \in \Delta_J$, we have} \quad   \C_J = \left\{ C_{\bm{\nu}} \ \vline \ \bm{\nu} \in \Delta_J \right\}.
    \end{equation}
In particular, for $J \in \{\varnothing, I\}$ the associated partitions are $\C_\varnothing = \{\PP\}$ and the set $\C_I = \binom{\PP}{1}$ of singleton subsets of $\PP$. In the case where $\I$ and all the $e_i=|\Delta_i|$ are finite, the partition $\C_J$ consists of $d_J$ classes of size $c_J$, where
    \begin{equation} \label{cd}
    c_J = |\Delta_{J^\c}| = \prod_{j \in J^\c} e_j \quad \text{and} \quad d_J = |\Delta_J| = \prod_{j \in J} e_j.
    \end{equation}

For $J\subseteq J'\subseteq I$, the projection $\pi_J$ induces a natural map $\Delta_{J'}\to \Delta_J$, namely
\begin{equation}\label{piJ}
\mbox{for $\bm{\nu}'=(\nu_i)_{i \in J'} \in \Delta_{J'}$, define}\quad  \bm{\nu}'\pi_{J}= \bm{\nu}  \quad \mbox{where $\bm{\nu} =(\nu_i)_{i \in J} \in \Delta_J$.} 
\end{equation}
Note that, for any $\bm{\delta} \in \PP$ such that $ \bm{\delta}\pi_{J'} = \bm{\nu}'$, we have $ \bm{\delta}\pi_{J'}\pi_J =\bm{\nu}'\pi_J= \bm{\nu}$ and $ \bm{\delta}\pi_J = \bm{\nu}$. In other words, 
\begin{equation}\label{eq:proj}
  \text{whenever }J\subseteq J'\subseteq I, \text{ the composition }\pi_{J'}\circ \pi_J =\pi_J.  
\end{equation}
Then, by \eqref{partition}, for each $ \bm{\nu}' \in \Delta_{J'}$, the class $C_{\bm{\nu}'}$ of $\C_{J'}$ is contained in the class  $C_{\bm{\nu}}$ of $\C_J$, where $\bm{\nu} := \bm{\nu}'\pi_J\in\Delta_J$, and this inclusion is strict if $J \subsetneq J'$ (since all the $e_j \geq 2$). It follows from this discussion that, if $J \subsetneq J'$ then, for $\bm{\nu}' \in \Delta_{J'}$ and $\bm{\nu} \in \Delta_J$, the $\C_{J'}$-class $C_{\bm{\nu}'}$ is contained in the $\C_J$-class $C_{\bm{\nu}}$ if and only if $\bm{\nu}'\pi_J = \bm{\nu}$, and the number of $\C_{J'}$-classes contained in a fixed $\C_J$-class is $|\Delta_{J' \setminus J}| = \prod_{j \in J' \setminus J} e_j$.

\smallskip
In summary, it follows easily from \eqref{partition} that 
\begin{itemize}
    \item $J\subseteq J'$ if and only if $\C_{J'}$ is a refinement of $\C_{J}$; 
    \item $J\subsetneq J'$ if and only if $\C_{J'}$ is a proper refinement of $\C_{J}$;
    \item  and for any $J, J' \subseteq I$, $\C_{J \cup J'} = \{ C \cap C' \ | \ C \in \C_J, \ C' \in \C_{J'} \}$. 
\end{itemize}    
We are particularly interested in the case of ancestral subsets of $I$, and we write, for $J, J' \in \A(\I)$,
    \begin{equation}\label{e:prec}
    \C_{J'} \prec \C_J \ \text{if and only if} \ \text{$\C_{J'}$ is a proper refinement of $\C_J$},   \ \text{or equivalently} \    J \subsetneq J'.
    \end{equation} 
Also we write $\C_{J'} \preccurlyeq \C_J$ if either $\C_{J'} \prec \C_J$ or $\C_{J'} = \C_J$.

Recalling that $A[j]$ is ancestral for each $j\in I$ and that, for each ancestral $J\in\A(\I)$ we have 
    \begin{equation} \label{CJ}
\mbox{$J = \bigcup_{j \in J} A[j]$,\quad  it follows that }\quad     \C_J = \Big\{ \bigcap_{j \in J} C_j \ \Big| \ C_j \in \C_{A[j]} \ \forall\,j \in J \Big\}. 
    \end{equation}
Note that the partition $\C_{A[j]}$ consists of singletons if and only if $j$ is the unique minimal element of $I$, and hence the only ancestral subset $J$ containing such an element $j$ is $J=A[j]=I$. This is the case, for example if $\I$ is a chain as in Figure~\ref{fig:posets-2}(a) or~\ref{fig:posets-3}(a), or the poset in Figure~\ref{fig:posets-3}(d). 

The pair $(\PP,\CC)$, where
    $ \CC := \{ \sim_J \ | \ J \in \A(\I)\}$,
was called a \emph{poset block structure} in \cite[Definition 4]{BPRS}. Alternatively, by identifying each equivalence relation $\sim_J$ with its set $\C_J$ of equivalence classes, the set $\CC$ in the poset block structure can be written as 
    \begin{equation} \label{posblkstr}
    \CC := \{ \C_J \ | \ J \in \A(\I)  \}.
    \end{equation}
The relation $\preccurlyeq$ on the set of point-partitions in \eqref{e:prec} is a partial order on $\CC$. 
In particular, for distinct $i, j \in I$,  the inclusion $A[i] \supsetneq A[j]$ holds if and only if $i \prec j$, and hence $\C_{A[i]} \prec \C_{A[j]}$ if and only if $i \prec j$. It follows that, 
    \begin{equation}\label{C*}
       \mbox{for}\quad  \CC^* := \{ \C_{A[i]} \ | \ i \in I \}, \quad \mbox{the `sub-poset'\quad $(\CC^*,\preccurlyeq)$ \quad is isomorphic to\quad $(I,\preccurlyeq)$}
    \end{equation} 
 under the map $i\to \C_{A[i]}$, for $i \in I$. Note that $\C_\varnothing\in\CC\setminus \CC^*$ for all poset block structures $(\PP, \CC)$. In addition $\CC\setminus \CC^*$ may contain other partitions $\C_J$, where $J$ is a nonempty ancestral subset    not of the form $A[i]$ for any $i\in I$. 
  This is the case, for example, for the posets in Figure~\ref{fig:posets-3}(c) and ~\ref{fig:posets-3}(d) where the subset $J=\{2,3\}$ is ancestral but not of the form  $A[i]$ for any $i\in I$, and so $\C_J\not\in\CC^*$.  In general,  if $\C_J\in\CC\setminus \CC^*$ with $J\neq\varnothing$, then by \eqref{CJ}, the classes of $\C_J$ are simply intersections of classes of some partitions $\C_{A[i]}$ in $\CC^*$, for example, those $A[i]$ such that $i\in J$ since $J=\cup_{i\in J} A[i]$. 

Illustrations for $\CC^*$ are provided in Figures \ref{fig:posets-2}, \ref{fig:posets-3}, and  \ref{fig:posets-4}, for $|I|=2,3$ and one example with $|I|=4$. A more detailed description of the partitions in $\CC$ will be provided in Section \ref{sec:ex} for those examples that are not chains or antichains.   

\begin{figure}[ht]
    
    \begin{subfigure}{\textwidth}
    \centering
    \includegraphics{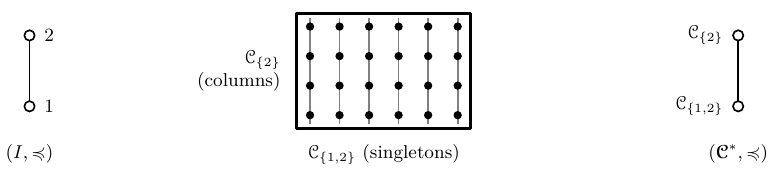}
    \caption{$\prod_{(I,\preccurlyeq)}(G_i,\Delta_i) \cong G_1 \wr G_2$}
    \label{fig:2ch}
    \end{subfigure}

    \vspace{14pt}
    \begin{subfigure}{\textwidth}
    \centering
    \includegraphics{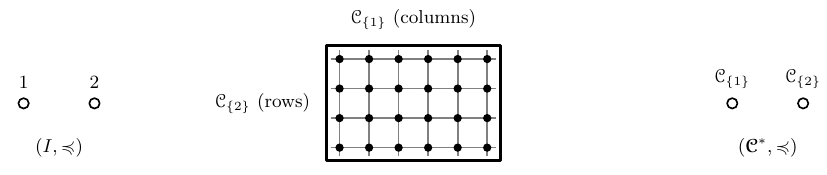}
    \caption{$\prod_{(I,\preccurlyeq)}(G_i,\Delta_i) \cong G_1 \times G_2$}
    \label{fig:2grid}
    \end{subfigure}
    
\caption{Poset block structures for partially ordered sets $(I,\preccurlyeq)$ with $|I| = 2$}
\label{fig:posets-2}
\end{figure}

\begin{figure}[h]

    \begin{subfigure}{\textwidth}
    \centering
    \includegraphics{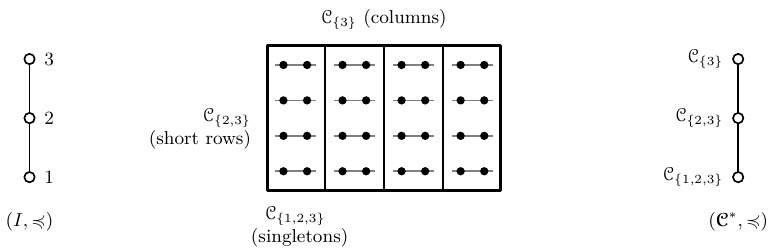}
    \caption{$\prod_{(I,\preccurlyeq)}(G_i,\Delta_i) \cong G_1 \wr G_2 \wr G_3$}
    \label{fig:3ch}
    \end{subfigure}

    \vspace{14pt}
    \begin{subfigure}{\textwidth}
    \centering
    \includegraphics{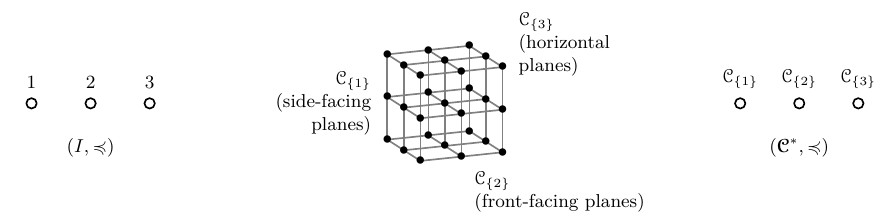}
    \caption{$\prod_{(I,\preccurlyeq)}(G_i,\Delta_i) \cong G_1 \times G_2 \times G_3$}
    \label{fig:3grid}
    \end{subfigure}    
    
\caption{Poset block structures for partially ordered sets $(I,\preccurlyeq)$ with $|I| = 3$; chain and antichain}
\label{fig:posets-3}
\end{figure}

\begin{figure}[ht] \ContinuedFloat
     \begin{subfigure}{\textwidth}
    \centering
    \includegraphics{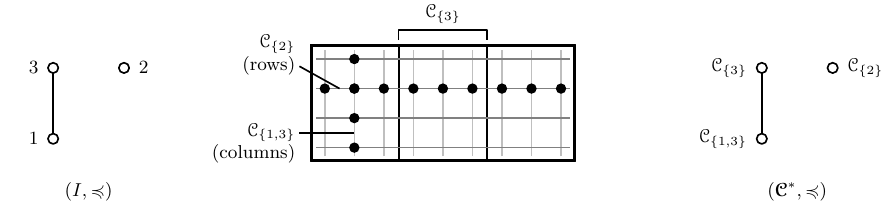}
    \caption{$\prod_{(I,\preccurlyeq)}(G_i,\Delta_i) \cong (G_1 \wr G_3) \times G_2$}
    \label{fig:ch-grid}
    \end{subfigure}
 
 \vspace{14pt}
    \begin{subfigure}{\textwidth}
    \centering
    \includegraphics{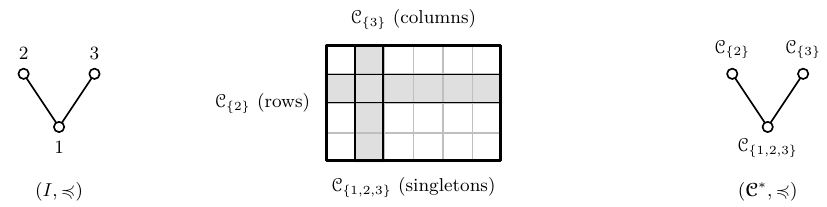}
    \caption{$\prod_{(I,\preccurlyeq)}(G_i,\Delta_i) \cong G_1 \wr (G_2 \times G_3)$}
    \label{fig:V}
    \end{subfigure}

    \vspace{14pt}
    \begin{subfigure}{\textwidth}
    \centering
    \includegraphics{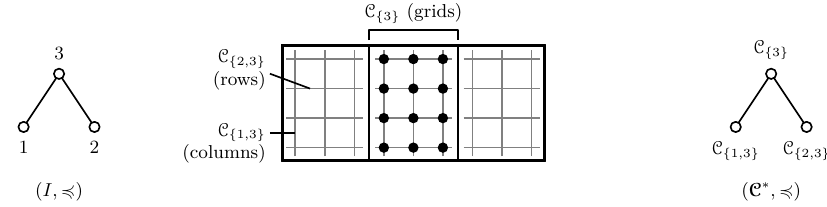}
    \caption{$\prod_{(I,\preccurlyeq)}(G_i,\Delta_i) \cong (G_1 \times G_2) \wr G_3$}
    \label{fig:V-inv}
    \end{subfigure}
    
\caption{Poset block structures for partially ordered sets $(I,\preccurlyeq)$ with $|I| = 3$; not chain or antichain}
\end{figure}

\begin{figure}[ht]
\centering
\includegraphics{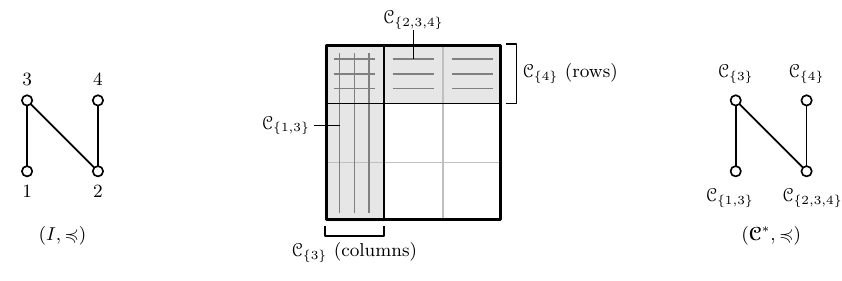}
\caption{Poset block structure for the `N' poset}
\label{fig:posets-4}
\end{figure}

\subsection{Generalised wreath products of groups}\label{ss:genwr}

Let $(\PP,\CC)$ be  a poset block structure relative to the poset $\I=(I,\preccurlyeq)$, with $\CC$ as in \eqref{posblkstr}  and point set $\PP=\Delta_I= \prod_{i \in I} \Delta_i$, and assume that $\I=(I,\preccurlyeq)$ satisfies the maximal condition (see Subsection~\ref{s:posets}). When we prove Theorem~\ref{t:gwp-design} both $\I$ and the $e_i$ will be finite, but for this section we do not assume finiteness.  For each $i \in I$  let $G_i \leq \Sym(\Delta_i) = S_{e_i}$ where $e_i=|\Delta_i|\geq2$, and denote by $F_i$ the set of all functions $f_i : \Delta_{A(i)} \to G_i$. Note that $F_i$ is a group under pointwise multiplication $\ast$ given by
     \begin{equation} \label{ptwise}
   \bm{\nu}(f_i \ast h_i) := \bm{\nu}f_i \cdot \bm{\nu}h_i \ \forall\,\bm{\nu} \in \Delta_{A(i)} \ \text{and} \ \forall\,f_i, h_i \in F_i,
    \end{equation} 
so $F_i$ is isomorphic to the cartesian product $G_i^{|\Delta_{A(i)}|}$. 
Consider 
\begin{equation}\label{F}
F := \prod_{i \in I} F_i = \prod_{i \in I} G_i^{|\Delta_{A(i)}|}\quad \mbox{with elements $f = (f_i)_{i \in I}$}.    
\end{equation}
As in \cite[Section 3]{BPRS} we identify $F$ with a subgroup of $\Sym(\PP)$: for each $f \in F$, define the map $f : \PP \to \PP$ (see  \cite[Section 3, Definition]{BPRS}) as follows:
     \begin{equation} \label{action}
\text{for each} \ \bm{\delta} = (\delta_i)_{i\in I} \in \PP, \quad    \bm{\delta} f := \bm{\varepsilon} = (\varepsilon_i)_{i\in I} \in \PP, \quad \text{where} \quad \varepsilon_i = \delta_i \big((\bm{\delta}\pi_{A(i)} )f_i\big) \ \forall\,i \in I. 
    \end{equation}
By \cite[Lemma 1]{BPRS}, distinct elements of $F$ determine distinct maps $\PP \to \PP$; since $\I$ satisfies the maximal condition it follows from \cite[Lemma 8]{BPRS} that the map $f$ in \eqref{action} is a bijection and hence lies in $\Sym(\PP)$; and by \cite[Theorem A]{BPRS} the subset $F$ of  $\Sym(\PP)$ is a subgroup. We note in passing that the maximality condition on $\I$ is necessary for the maps $f$ to be invertible,  see \cite[Example 2]{BPRS} for an example where  $\I$ contains an infinite ascending chain.

This subgroup is  called the \emph{generalised wreath product} of the permutation groups $G_i$ on $\Delta_i$, and as in \cite{BPRS} we denote it by $F=\prod_{(I,\preccurlyeq)}(G_i,\Delta_i)$ (see the discussion following \cite[Theorem A]{BPRS}).

For each $J\subseteq I$, recall that $\pi_J:\PP\to \Delta_J$ denotes the natural projection map. We also write $F_J=\prod_{j\in J} F_j$ and $\varphi_J:F\to F_J$ for the natural projection $\varphi_J:(f_i)_{i\in I}\to (f_j)_{j \in J}$.  Suppose that $J$ is an ancestral subset of $I$. Then by the discussion after \cite[Lemma 2]{BPRS}, the action \eqref{action} of $F$ on $\PP$ induces an action of $F_J=F\varphi_J$ on $\Delta_J$, namely,  for $f \in F$ and $\bm{\nu}=(\nu_j)_{j \in J} \in \Delta_{J}$,  choose a point $\bm{\delta}\in \PP$ such that $\bm{\delta} \pi_{J} = \bm{\nu}$ and define
    \begin{equation} \label{action-J}
    \bm{\nu} (f \varphi_J) = (\nu_j)_{j \in J} (f \varphi_J) = ({\bm{\delta}} f) \pi_J,
    \end{equation}
    noting that this is independent of the choice of $\bm{\delta}$. By \cite[Theorem A]{BPRS} (and the lemmas leading up to it),  this $F_J$-action is faithful and the induced group on $\Delta_J$ is the generalised wreath product $F_J=\prod_{(J,\preccurlyeq)}(G_j,\Delta_j)$. Further, again by \cite[Lemma 2]{BPRS}, the partition $\CC_J$ is $F$-invariant and by \cite[Theorem A (ii)]{BPRS}, the induced group $F^{\C_J}$ on $\C_J$ is permutationally isomorphic to $F_J$ on $\Delta_J$, and the kernel $F_{(\C_J)}$ of this $F$-action is
        \begin{equation} \label{ker}
    F_{(\C_J)} = \{ f \in F \ | \ f_j = \iota_j\in F_j, \ \forall\,j \in J \},\quad \mbox{where $\iota_j:\bm{\nu}\to 1_{G_j}$ for all }\bm{\nu}\in \Delta_{A(j)}
    \end{equation}
By \eqref{ptwise},  $\iota_j$ is the identity of the group $F_j$. Thus $F_{(\C_J)}$ may be identified with $F_{J^\c} = \prod_{j \in J^\c} F_j$. 

By \cite[Theorem B]{BPRS}, the full automorphism group of the poset block structure $(\PP,\CC)$, with $\CC$ as in \eqref{posblkstr}, is the generalised wreath product $\prod_{(I,\preccurlyeq)} (S_{e_i},\Delta_i)$ (taking the group $G_i=\Sym(\Delta_i)$ for each $i\in I$). As discussed at the end of Section~\ref{s:posetbkstr}, each $\C_J\in\CC\setminus \CC^*$, with  $\CC^*$ as in \eqref{C*}, corresponds to an ancestral set $J=\cup_{j\in J} A[j]$, and if a permutation preserves each of the partitions $\C_{A[j]}$, for $j\in J$, then it also preserves $\C_J$. Therefore a permutation preserving each partition in $\CC^*$ also preserves all the partitions in $\CC$, that is, 
$\prod_{(I,\preccurlyeq)} (S_{e_i},\Delta_i)$ is also the full automorphism group of the substructure `$(\PP,\CC^*)$'.  

We finish this subsection with examples for the two posets $\I=(I,\preccurlyeq)$ with $|I|=2$. More general versions of these examples will be discussed in Subsection~\ref{s:dirpdt}.  

\begin{example} \label{poset:grid}
Let $(I,\preccurlyeq)$ be the antichain $\{1,2\}$ (see Figure~\ref{fig:2grid}), and let $G_i\leq \Sym(\Delta_i)$ for $i=1,2$. Then $A(1) = A(2) = \varnothing$, and thus the Cartesian products $\Delta_{A(1)}$ and $\Delta_{A(2)}$ are singletons. More significantly, there are two proper nonempty ancestral subsets, namely, $A[1] = \{1\}$ and $A[2] = \{2\}$, and it follows from the discussion above that $F_i=G_i$ for each $i$, and the generalised wreath product $F=\prod_{(I,\preccurlyeq)} (G_i,\Delta_i) \cong G_1 \times G_2$.  The associated partitions of $\PP=\Delta_1\times \Delta_2$ are $\C_{A[1]} = \big\{ \{ \bm{\delta}\in\PP \ | \ \delta_1 = \nu_1 \} \ | \ \nu_1 \in \Delta_1 \big\}$ and $\C_{A[2]} = \big\{ \{ \bm{\delta}\in\PP \ | \ \delta_2 = \nu_2 \} \ | \ \nu_2 \in \Delta_2 \big\}$. We can visualise the poset block structure by arranging the points of $\PP$ in an $e_1 \times e_2$ grid, with the points $\bm{\delta}$ having the same $\delta_1$-coordinate lying in the same row, and those with the same $\delta_2$-coordinate lying in the same column. Then $\C_{A[1]}$ is the set of rows and $\C_{A[2]}$ is the set of columns, as shown in Figure \ref{fig:2grid}. This poset block structure is also called a \emph{grid} in \cite{grids22,Linsp09grid,BMSVV}.
\end{example}

\begin{example} \label{poset:2ch}
Let $(I,\preccurlyeq)$ be the chain $1 \prec 2$ (see Figure \ref{fig:2ch}), and let $G_i\leq \Sym(\Delta_i)$ for $i=1,2$. Then $A(2) = \varnothing$ is ancestral, and there is a unique proper nonempty ancestral subset, namely $A[2] =A(1) = \{2\}$, and the associated partition  of $\PP=\Delta_1\times \Delta_2$ is $\C_{A[2]} = \big\{ \{ \bm{\delta} \in\PP\ | \ \delta_2 = \nu_2 \} \ | \ \nu_2 \in \Delta_2 \big\}$.
It follows from the discussion above that  $F_1=G_1^{|\Delta_2|}$, $F_2=G_2$, and the generalised wreath product $F=\prod_{(I,\preccurlyeq)} (G_i,\Delta_i) \cong G_1 \wr G_2$.  The point set $\PP$ can again be visualised as the $e_1 \times e_2$ grid in Example \ref{poset:grid}, with the partition $\C_{A[2]}$ being the set of columns.
\end{example}

\subsection{Orbitals and orbits on unordered pairs}\label{s:orbs}

Let $(\PP,\CC)$ be  a poset block structure relative to $\I=(I,\preccurlyeq)$, with $\CC$ as in \eqref{posblkstr}  and point set $\PP=\Delta_I= \prod_{i \in I} \Delta_i$, such that $\I=(I,\preccurlyeq)$ satisfies the maximal condition (see Subsection~\ref{s:posets}). For each $i \in I$  let $G_i \leq \Sym(\Delta_i) = S_{e_i}$ where $|\Delta_i|=e_i\geq2$, and let $F=\prod_{(I,\preccurlyeq)} (G_i,\Delta_i)$ be the corresponding generalised wreath product as in \eqref{F}. In this subsection we assume that  each $G_i$ is $2$-transitive on $\Delta_i$, and we describe the $F$-orbits on ordered pairs (orbitals) and unordered pairs from $\PP$.

Recall that the border $\Bdy{J}$ of an ancestral subset $J\subseteq $ is the set of all maximal elements in $J^\c$. For example, for each $i\in I$, the border $\Bdy{A(i)}$ contains $i$, and by convention $\Bdy{I} = \varnothing$.
For each $i \in I$ let $D_i := \{ (\delta,\delta) \ | \ \delta \in \Delta_i \}$ and $E_i := \Delta_i \times \Delta_i$, and for any $J \subseteq I$ and $\Gamma_i \subseteq E_i$ define
	\[ 
    \otimes_{i \in J} \Gamma_i := \{ (\bm{\delta}, \bm{\varepsilon}) \in \PP \times \PP \ | \ \forall\,i \in J, \ (\delta_i, \varepsilon_i) \in \Gamma_i \}. \]
Since  $G_i$ is $2$-transitive on $\Delta_i$ for each $i$, it follows from \cite[Theorem C]{BPRS} that the $F$-orbits in $\PP \times \PP$ are precisely the sets $O_J$, for each $J \in \A(\I)$, where
	 \[ O_J = \left( \otimes_{i \in J} D_i \right) \otimes \left( \otimes_{i \in \Bdy{J}} (E_i \setminus D_i) \right) \otimes \left( \otimes_{i \in (J \cup \Bdy{J})^\c} E_i \right) 
     \]
(and this is not true if some $G_i$ is not $2$-transitive). Equivalently,
	\begin{align}
	O_J
	&= \left\{  (\bm{\delta}, \bm{\varepsilon})  \ | \ \delta_i = \varepsilon_i \ \forall\,i \in J \ \text{and} \ \delta_i \neq \varepsilon_i \ \forall\,i \in \Bdy{J} \right\} \label{O_J} \\
	&= \left\{ (\bm{\delta}, \bm{\varepsilon})  \ | \ \bm{\delta}\pi_J = \bm{\varepsilon} \pi_J \right\} \setminus \bigcup_{i \in \Bdy{J}} \left\{ (\bm{\delta}, \bm{\varepsilon}) \ | \ \bm{\delta} \pi_{J \cup \{i\}} = \bm{\varepsilon} \pi_{J \cup \{i\}} \right\} \notag \\
	&= \left\{  (\bm{\delta}, \bm{\varepsilon})  \ | \ \bm{\delta}, \bm{\varepsilon} \ \text{belong to the same $\C_J$-class but lie in different $\C_{J \cup \{i\}}$-classes} \right. \notag \\
	&\phantom{=}\;\, \left. \text{for all $i \in \Bdy{J}$} \right\}. \notag
	\end{align}
In particular, as noted above, $J = I$ is ancestral and $\Bdy{I} = \varnothing$, so $O_I = \{ (\bm{\delta}, \bm{\delta}) \mid \bm{\delta} \in \PP \}$, the diagonal orbital. Further, for each ancestral subset $J\subseteq I$, it follows from the expression for $O_J$ in \eqref{O_J} that $O_J$ is self-paired, that is to say, $(\bm{\delta}, \bm{\varepsilon})\in O_J$ if and only if $(\bm{\varepsilon}, \bm{\delta})\in O_J$. Thus the $F$-orbits on unordered pairs of distinct points are the sets 
\[
Q_J:= \{\ \{ \bm{\delta}, \bm{\varepsilon}\} \ \mid (\bm{\delta}, \bm{\varepsilon})\in O_J\}\quad \mbox{for proper ancestral subsets $J\subset I$, and $|Q_J|=|O_J|/2$.}
\]
Finally, in the case where $\I$ and $\PP$ are finite, we determine the cardinality of the orbital $O_J$ for each proper ancestral subset $J$, noting that $|O_I|=|\PP|$.

\begin{proposition} \label{p:orbit}
Suppose that $\I=(I,\preccurlyeq)$ is a finite poset, each $e_i=|\Delta_i|\geq 2$ is finite, and let $v := |\PP|=\prod_{i\in I} e_i$. Let $J \in \A(\I)\setminus\{I\}$ (possibly $J = \varnothing$), and let $O_J$ be as in \eqref{O_J}. Then
	 \begin{equation}\label{|O_J|}
	|O_J|
	= v \left( \prod_{i \in \Bdy{J}} (e_i - 1) \right) \left( \prod_{i \in (J \cup \Bdy{J})^\c} e_i \right)
	= v \sum_{S \subseteq \Bdy{J}} \Big( (-1)^{|S|} \cdot\prod_{i \in (J \cup S)^\c} e_i\Big).
	\end{equation}
\end{proposition}

\begin{proof}
It follows from the first expression in \eqref{O_J} that
	\[
	|O_J|
	= \left( \prod_{i \in J} e_i \right) \left( \prod_{i \in \Bdy{J}} e_i(e_i - 1) \right) \left( \prod_{i \in (J \cup \Bdy{J})^\c} e_i^2 \right) 
	= v \left( \prod_{i \in \Bdy{J}} (e_i - 1) \right) \left( \prod_{i \in (J \cup \Bdy{J})^\c} e_i \right)
	\]
	with the convention that $\prod_{i \in \varnothing} e_i = \prod_{i \in \varnothing} (e_i-1)= 1$. Also we may use the second line of \eqref{O_J} to obtain the second expression for $|O_J|$. It follows from the Principle of Inclusion--Exclusion \cite[Chapter IV.2B Sieve Formulae]{Aig} that 
	\begin{align*}
	\Big|\bigcup_{i \in \Bdy{J}} \left\{  (\bm{\delta}, \bm{\varepsilon})  \ | \ \bm{\delta}\pi_{J \cup \{i\}} = \bm{\varepsilon} \pi_{J \cup \{i\}}  \right\}\Big|
	&= \sum_{\varnothing\ne S \subseteq \Bdy{J}} (-1)^{|S|-1} \Big| \bigcap_{i \in S} \{  (\bm{\delta}, \bm{\varepsilon})  \ | \ \bm{\delta}\pi_{J \cup \{i\}} = \bm{\varepsilon} \pi_{J \cup \{i\}}  \} \Big| \\
	&= \sum_{\varnothing\ne S \subseteq \Bdy{J}} (-1)^{|S|-1} \left| \{  (\bm{\delta}, \bm{\varepsilon})  \ | \ \bm{\delta}\pi_{J \cup S} = \bm{\varepsilon} \pi_{J \cup S}  \} \right| \\
	&= \sum_{\varnothing \ne S \subseteq \Bdy{J}} (-1)^{|S|-1} \cdot \Big(|\PP| \prod_{i \in (J \cup S)^\c} e_i\Big).
	\end{align*}
Moreover, the cardinality of $\left\{  (\bm{\delta}, \bm{\varepsilon})  \ | \ \bm{\delta}\pi_J = \bm{\varepsilon} \pi_J \right\}$ is $|\PP|\cdot \prod_{i \in (J \cup S)^\c} e_i$ with $S=\varnothing$. Thus we have 
	\begin{align}
	|O_J|
	&= \left|\left\{  (\bm{\delta}, \bm{\varepsilon})  \ | \ \bm{\delta}\pi_J = \bm{\varepsilon} \pi_J  \right\}\right| - \Big|\bigcup_{i \in \Bdy{J}} \left\{  (\bm{\delta}, \bm{\varepsilon})  \ | \ \bm{\delta}\pi_{J \cup \{i\}} = \bm{\varepsilon} \pi_{J \cup \{i\}}  \right\}\Big| \notag \\
	&= |\PP|\cdot \prod_{i \in (J \cup \varnothing)^\c} e_i - \sum_{\varnothing\ne S \subseteq \Bdy{J}} (-1)^{|S|-1} \cdot \Big(|\PP| \prod_{i \in (J \cup S)^\c} e_i\Big) \notag \\
		&= |\PP|\cdot  \sum_{S \subseteq \Bdy{J}} \Big( (-1)^{|S|} \cdot\prod_{i \in (J \cup S)^\c} e_i\Big) = v \sum_{S \subseteq \Bdy{J}} \Big( (-1)^{|S|} \cdot\prod_{i \in (J \cup S)^\c} e_i\Big) \label{eq:O_J2}
	\end{align}
which completes the proof.
\end{proof}

\subsection{Direct products of poset block structures}\label{s:dirpdt}

We generalise the observations in Examples~\ref{poset:grid} and~\ref{poset:2ch} to the cases where the elements $1$ and $2$ of $I$ are replaced with disjoint nonempty subsets $I_1$ and $I_2$. We continue to assume that the poset $\I=(I,\preccurlyeq)$ satisfies the maximal condition.
In this section suppose that  $I$ can be written as the disjoint union $I = I_1 \,\dot\cup\, I_2$  of two nonempty subsets such that:
    \begin{equation} \label{hyp:dir}
   \text{ \emph{there do not exist} $i_1\in I_1, i_2\in I_2$ with $i_1\prec i_2$ or $i_2\prec i_1$. }
    \end{equation}
Then $\PP = \Delta_{I_1} \times \Delta_{I_2}$ and, as noted in \cite[p. 74]{BPRS}, for subgroups $G_i\leq \Sym(\Delta_i)$ ($i\in I$), the generalised wreath product  $F = \prod_{(I,\preccurlyeq)} (G_i,\Delta_i)$ is the permutation direct product of the generalised wreath products $\ff[1] = \prod_{(I_1,\preccurlyeq)} (G_i,\Delta_i)$ and $\ff[2] = \prod_{(I_2,\preccurlyeq)} (G_i,\Delta_i)$. That is to say, $F = \ff[1] \times \ff[2]$. We call the poset block structure $(\PP,\CC)$ the \emph{direct product} of the poset block structures  $\big(\Delta_{I_1}, \cc[1]\big)$ and $\big(\Delta_{I_2}, \cc[2]\big)$, where, for each $i \in \{1,2\}$, $\cc[i] = \{ \C_J \ | \ J \in \A(I_i) \}$. We give more details and in particular we specify the ancestral sets explicitly. 

It follows from \eqref{hyp:dir} that the set $\mathcal{A}(I)$ of ancestral subsets of $I$ is as follows  (see Subsection~\ref{s:posets}):
    \begin{equation}\label{e:ancdir}
    \mathcal{A}(I) = \{ J_1 \ \dot\cup \ J_2 \ \mid \ J_1 \in \mathcal{A}(I_1), J_2 \in \mathcal{A}(I_2) \}. 
    \end{equation}
In particular, $I_1$ and $I_2$ are ancestral subsets of $I$. Also, for $j \in I_i$ we have $A(j)\subseteq I_i$, and hence $F_j$, as defined in \eqref{ptwise}, is the same  subgroup of $F$ and $\ff[i]$. For each  ancestral subset $J = J_1 \dot \cup J_2$ with each $J_i = J \cap I_i$ as in \eqref{e:ancdir}, the corresponding subgroup of $F$ is therefore: 
    \[
    F_J = \prod_{j\in J} F_j = \left( \prod_{j \in J_1} F_j \right) \times \left( \prod_{j \in J_2} F_j \right) = \ff[1]_{J_1} \times \ff[2]_{J_2}.
    \]
In particular, since $F_\varnothing = 1$ and each $\ff[i]_\varnothing = 1$,  taking $J=I_i$ we have $F_{I_i} = \ff[i]_{I_i}\times 1 =\ff[i]$, and taking $J=I$ yields $F=F_I= \ff[1]\times \ff[2]$.  

We also examine the partition $\C_{I_i}$ of $\PP$. By \eqref{partition}, the classes are $C_{\bm \nu} = \{ {\bm\delta} \in \PP \ | \ {\bm\delta}\pi_{I_i} = {\bm\nu} \}$, for ${\bm\nu} \in \Delta_{I_i}$. Clearly $\ff[3-i]$ fixes each $C_{\bm\nu}$ setwise and the $\ff[3-i]$-action on $C_{\bm \nu}$ is equivalent to its action on $\Delta_{I_{3-i}}$. Also the group $\ff[i]$ permutes the classes of $\C_{I_i}$, and the $\ff[i]$-action on $\C_{I_i}$ is equivalent to its action on $\Delta_{I_{i}}$. We summarise these remarks in the following theorem. 
\begin{theorem} \label{thm:dir}
Suppose that the poset $(I,\preccurlyeq)$ satisfies the maximal condition and that $I$ is the disjoint union $I = I_1 \,\dot\cup\, I_2$ of two non-empty subsets such that  \eqref{hyp:dir} holds. Then the set of ancestral subsets of $I$ is as in \eqref{e:ancdir}, $F = F_{I_1} \times F_{I_2} = \ff[1] \times \ff[2]$, and moreover:
\begin{enumerate}[(a)]
    \item $I_1$ and $I_2$ are ancestral subsets of $I$, and for $i \in \{1,2\}$, the partition $\C_{I_i}$ of $\PP$, with classes as in \eqref{partition}, is $F$-invariant and $F^{\C_{I_i}}$ is permutationally isomorphic to the $\ff[i]$-action on $\Delta_{I_i}$; 
    \item and for $i \in \{1,2\}$,  the group induced on each class $C \in \C_{I_i}$ by its setwise stabiliser $F_C$ is permutationally isomorphic to the $\ff[3-i]$-action on $\Delta_{I_{3-i}}$.
\end{enumerate}
\end{theorem}

\subsection{Kronecker products of poset block structures}\label{s:kron}

Suppose now that the poset $(I,\preccurlyeq)$ satisfies the maximal condition and that $I$  can be written as the disjoint union $I = I_1 \,\dot\cup\, I_2$  of two nonempty subsets such that:
    \begin{equation} \label{hyp:wr}
  \text{ \emph{for all $i_1 \in I_1$ and $i_2 \in I_2$\quad we have} } \quad i_1 \prec i_2.
    \end{equation}
Then again $\PP = \Delta_{I_1} \times \Delta_{I_2}$ 
 and, as noted in \cite[p. 74]{BPRS}, {for subgroups $G_i\leq \Sym(\Delta_i)$ ($i\in I$), the generalised wreath product  $F = \prod_{(I,\preccurlyeq)}(G_i,\Delta_i)$ is the permutation wreath product of the generalised wreath products $\ff[1]=\prod_{(I_1,\preccurlyeq)}(G_i,\Delta_i)$ and $\ff[2]=\prod_{(I_2,\preccurlyeq)}(G_i,\Delta_i)$}. That is to say, $F=\ff[1]\wr \ff[2]$.
We call the poset block structure $(\PP,\CC)$ the \emph{Kronecker product} of the poset block structures  $(\Delta_{I_1}, \CC_{I_1})$ and   $(\Delta_{I_2}, \CC_{I_2})$. We give more details and in particular specify the ancestral sets explicitly. 

It follows from \eqref{hyp:wr} that the set  $\mathcal{A}(I)$ of ancestral subsets of $I$ is as follows  (see Subsection~\ref{s:posets}):
\begin{equation}\label{e:ancwr}
    \mathcal{A}(I) = \{ J \ \dot\cup \ I_2 \ \mid \ J \in \mathcal{A}(I_1), J \ne \varnothing \} \ \dot\cup \ \mathcal{A}(I_2). 
\end{equation}
In particular,  $I_2$ is an ancestral subset of $I$, and for each $j \in I_1$, we have $I_2 \subseteq A(j)$ and $A'(j) := A(j) \cap I_1$ is the set of all `ancestors' of $j$ in the sub-poset $(I_1,\preccurlyeq)$. Moreover $\Delta_{A(j)} = \Delta_{A'(j)} \times \Delta_{I_2}$. Thus
    \[ 
    \ff[1] = \prod_{(I_1,\preccurlyeq)} (G_i,\Delta_i) = \prod_{j \in I_1} G_j^{|\Delta_{A'(j)}|},\quad \mbox{and}\quad \ff[2] =  \prod_{(I_2,\preccurlyeq)} (G_i,\Delta_i) = \prod_{j \in I_2} G_j^{|\Delta_{A(j)}|} = F_{I_2} < F. 
    \]
Even though $I_1$ is not an ancestral subset of $I$, writing $F_{I_1} := \prod_{j \in I_1} F_j$ as in \eqref{F}, we have
    \[ 
    F_{I_1} = \prod_{j \in I_1} F_j = \prod_{j \in I_1} G_j^{|\Delta_{A(j)}|} = \prod_{j \in I_1} G_j^{|\Delta_{A'(j)} \times \Delta_{I_2}|} = \left( \prod_{j \in I_1} G_j^{|\Delta_{A'(j)}|} \right)^{|\Delta_{I_2}|} = \left( \ff[1] \right)^{|\Delta_{I_2}|}, 
    \]
    and as a set $F$ is the cartesian product
    \[
    F=  \prod_{j \in I} F_j = \left( \ff[1] \right)^{|\Delta_{I_2}|} \times \prod_{j \in I_2} F_j =  \left( \ff[1] \right)^{|\Delta_{I_2}|} \times \ff[2].
    \]
Now we examine the partition $\C_{I_2}$ of $\PP$. By \eqref{partition}, the classes are $C_{\bm{\nu}} =\{ {\bm \delta} \in \PP \ \mid \ {\bm\delta}\pi_{I_2} = {\bm \nu} \}$, for ${\bm \nu}\in  \Delta_{I_2}$. The group induced by $F$ on $\C_{I_2}$ is permutationally isomorphic to the action of $\ff[2] = F_{I_2}$ on $\Delta_{I_2}$, and the kernel of this $F$-action is $F_{I_1}$. To examine the action of this kernel, let $f \in F_{I_1}$,  $C_{\bm{\nu}} \in \C_{I_2}$, and ${\bm \delta} = (\delta_i)_{i\in I} \in C_{\bm{\nu}}$, so that $\delta_i = \nu_i$, and  $f_i = \iota_i$, for all $i\in I_2$. Also let ${\bm\varepsilon} = {\bm \delta} f$. Then it follows from the definition of the action \eqref{action} that, for each  $i \in I_2$,
    \[ 
    \varepsilon_i = \delta_i \left( {\bm \delta}\pi_{A(i)} f_i \right) = \delta_i \left(  {\bm \delta}\pi_{A(i)} \iota_i \right) = \delta_i 1_{G_i} = \delta_i = \nu_i 
    \]
so ${\bm \varepsilon} = {\bm \delta} f\in C_{\bm{\nu}}$. On the other hand, for $i \in I_1$, and using the fact that $A'(i) := A(i) \cap I_1$, we have
    \[ 
    \varepsilon_i = \delta_i \left(  {\bm \delta}\pi_{A(i)} f_i \right) = \delta_i \left(  {\bm \delta}\pi_{A'(i)} f'_i \right)  \quad \text{for some $f'_i \in \ff[1]$}, 
    \]
and each element of the group $\ff[1]$ arises for different choices of $f$ in $F_{I_1}$. Thus the group induced by $F_{I_1}$ on $C_{\bm{\nu}}$ is permutationally isomorphic to the $\ff[1]$-action on $\Delta_{I_1}$. Since $F_{I_1} \cong (\ff[1])^{|\Delta_{I_2}|}$, it follows that $F \cong \ff[1] \wr \ff[2]$ with the induced group $F^{\C_{I_2}}$ permutationally isomorphic to $(\ff[2])^{\Delta_{I_2}} = (F_{I_2})^{\Delta_{I_2}}$, and the group induced on each class of $\C_{I_2}$ permutationally isomorphic to $(\ff[1])^{\Delta_{I_1}}$. We summarise these observations in Theorem \ref{thm:wr}, using this notation.

\begin{theorem} \label{thm:wr}
Suppose that the poset $(I,\preccurlyeq)$ satisfies the maximal condition and that $I$ is the disjoint union $I = I_1 \,\dot\cup\, I_2$ of two non-empty subsets such that \eqref{hyp:wr} holds. Then  the set of ancestral subsets of $I$ is as in \eqref{e:ancwr},  $F \cong \ff[1] \wr \ff[2]$, and moreover:
\begin{enumerate}[(a)]
    \item $I_2$ is an ancestral subset of $I$,  the partition $\C_{I_2}$ of $\PP$ with classes as in \eqref{partition} is $F$-invariant, and $F^{\C_{I_2}}$ is permutationally isomorphic to the $\ff[2]$-action on $\Delta_{I_2}$; and 
    \item for $j \in I_1$, the ancestral set $A(j) = A'(j) \cup I_2$, where $A'(j) = A(j) \cap I_1$ is the ancestral set for $j$ relative to the sub-poset $(I_1,\preccurlyeq)$, and the  group induced on each class $C \in \C_{I_2}$ by its setwise stabiliser $F_C$ is permutationally isomorphic to the $\ff[1]$-action on $\Delta_{I_1}$.
\end{enumerate}
\end{theorem}

It follows from Theorems~\ref{thm:dir} and \ref{thm:wr} that the groups $\prod_{(I,\preccurlyeq)}(G_i,\Delta_i)$ for the posets in  Figures \ref{fig:posets-2}, \ref{fig:posets-3}  are as given in the captions.

\section{Array function of a block}

Let $B\subseteq \PP$. Let $J$ be an ancestral subset and $\bm{\nu} \in \Delta_J$. 
 It follows from the definition \eqref{arrayfn} of the array function $\chi_B$ and the definition \eqref{partition} of the class $C_{\bm{\nu}}$ that
    \begin{equation}
    (\bm{\nu})\chi_B = |B \cap C_{\bm{\nu}}| \ \ \text{for every} \ C_{\bm{\nu}} \in \C_J.
    \end{equation}
This function describes how the points of $B$ are distributed among the parts of each partition $\C_J$ where $J$ is ancestral.

In particular, $J = I$ is an ancestral subset, with $\Delta_I = \PP$, so
    \begin{equation} \label{eq:array-I}
    \mbox{if $\bm{\delta} \in \PP = \Delta_I$\;\; then} \;\;
    (\bm{\delta})\chi_B =
        \begin{cases}
        1 &\text{if $\bm{\delta} \in B$} \\
        0 &\text{otherwise}.
        \end{cases}
    \end{equation}    
On the other hand, for $J = \varnothing$, recall that $\Delta_\varnothing$ is a singleton set and $\C_\varnothing = \{\PP\}$, so 
    \begin{equation} \label{eq:array-empty}
     \mbox{if $\bm{\nu} \in \Delta_\varnothing$\;\; then} \;\;
     (\bm{\nu})\chi_B = |B \cap \PP| = |B| = k.
    \end{equation}

Note that Equation \eqref{2des} involves the sum $\sum_{\bm{\nu} \in \Delta_{J \cup S}}  \left((\bm{\nu})\chi_B \right)^2$, where for every choice of $S$ the set $J\cup S$ is itself ancestral. Moreover, the formulae for various ancestral sets $J$ might use the same sums (as we will see in the Section \ref{sec:ex}).
It is thus convenient to introduce the function $\mu_B$ from $\A(\I)$ to $\mathbb{Z}$, defined by
\begin{equation}\label{eq:sum-squares}
\Sq[J]=\sum_{\bm{\nu} \in \Delta_{J}}  \left((\bm{\nu})\chi_B \right)^2.
\end{equation}
We have the following lemma. 

\begin{lemma} \label{lem:array-sqs}
    For any $k$-subset $B$ of $\PP = \Delta_I$, and $\mu_B$ as in \eqref{eq:sum-squares},
        \[ 
        \Sq[I] = k \quad \mbox{and} \quad \Sq[\varnothing] = k^2. 
        \]
\end{lemma}
\begin{proof}
It follows from \eqref{eq:array-I} and \eqref{eq:sum-squares} that
    $\Sq[I] = \sum_{\bm{\delta} \in B} 1^2 =  k $, and from \eqref{eq:array-empty} that
    $\Sq[\varnothing] = \sum_{\bm{\nu} \in \Delta_\varnothing} k^2 = k^2$. \qedhere 
\end{proof}

\section{Proof of Theorem \ref{t:gwp-design}} \label{s:proof}

We prove Theorem~\ref{t:gwp-design} in two stages. First we formalise the hypotheses of this result.

\begin{hypothesis}\label{h:2des}
    Let $(I,\preccurlyeq)$ be a finite partially ordered set  with $|I|\geq 2$, and  for each $i \in I$ let $\Delta_i$ be a finite set of size $|\Delta_i|=e_i \geq 2$, and let $G_i \leq \Sym(\Delta_i)$ such that $G_i$ is $2$-transitive on $\Delta_i$. Let $\PP = \prod_{i \in I} \Delta_i$, of size $v = \prod_{i \in I} e_i$, and let $F = \prod_{(I,\preccurlyeq)} (G_i,\Delta_i)$ be the generalised wreath product as described in Section~\ref{ss:genwr}. Let $B$ be a $k$-subset of $\PP$, let $\B := B^F$, and let $\D = (\PP,\B)$ be the corresponding incidence structure. 
\end{hypothesis}

Next we prove a  slightly weaker version of Theorem~\ref{t:gwp-design}. The only differences are that Proposition \ref{p:2des} involves the function $\mu_B$ from \eqref{eq:sum-squares} while Theorem~\ref{t:gwp-design} involves the array function $\chi_B$, and that Proposition \ref{p:2des} includes a condition for $J=\varnothing$. Recall that for any ancestral subset $J$ and any subset $S \subset \Bdy{J}$, the union $J \cup S$ is also an ancestral subset.

\begin{proposition} \label{p:2des} 
Suppose that Hypothesis~\ref{h:2des} holds, and let $\mu_B$ be as  in \eqref{eq:sum-squares}. Then $\D = (\PP,\B)$ is a $2$-design if and only if the following condition holds for each proper ancestral subset $J$ of $I$:
    \begin{equation} \label{2deswithmu}
       \sum_{S \subseteq \Bdy{J}}(-1)^{|S|} \Sq[J\cup S]
    = \frac{k(k-1)}{v-1} \left( \prod_{i \in \Bdy{J}} (e_i - 1) \right) \left( \prod_{j \in (J \cup \Bdy{J})^\c} e_j \right),
    \end{equation}
where $J^\c$ denotes the complement of $J$ in $I$, and $\Bdy{J}$ is the set of all maximal elements in $J^\c$.
\end{proposition}

\begin{proof}
By \cite[Theorem C]{BPRS}, the orbits of the group $F$ on $\PP\times \PP$ are precisely the sets $O_J$ defined in \eqref{O_J}, for each ancestral subset $J$ of $I$ (including the set $J=I$). Moreover, it follows from \eqref{O_J} that each of these sets $O_J$ is self-paired, and  (as we observed above) the set $O_I=\{(\bm{\delta}, \bm{\delta})\mid \bm{\delta}\in\PP\}$ is the diagonal orbit. Thus the $F$-orbits on unordered pairs of distinct points of $\PP$ are the sets $Q_J:=\{\{\bm{\delta},\bm{\varepsilon}\}\mid (\bm{\delta},\bm{\varepsilon})\in O_J\}$, for the proper ancestral subsets $J$. Note that $|O_J|=2 |Q_J|$ for each such $J$.

Therefore, by the criteria given in  \cite[Proposition 1.3]{CP93}, $\D = (\PP,\B)$ is a $2$-design  if and only if there is a constant $c$ such that, for each proper ancestral subset $J$, 
    \begin{equation*}
    p_J := \frac{|\{ \{\bm{\delta},\bm{\varepsilon}\} \ | \ \bm{\delta}, \bm{\varepsilon} \in B; \  \{\bm{\delta},\bm{\varepsilon}\} \in Q_J \}|}{|Q_J|} = \frac{|\{ (\bm{\delta},\bm{\varepsilon}) \ | \ \bm{\delta}, \bm{\varepsilon} \in B; \  (\bm{\delta},\bm{\varepsilon}) \in O_J \}|}{|O_J|}\ \ \mbox{is equal to $c$.}
    \end{equation*}
Assuming that there is such a constant, that is, $p_J=c$ for all $J\ne I$, it is readily checked, by summing  $p_J \cdot |O_J|$ over $J\ne I$, that
	\begin{align*}
	cv(v-1) &= c\sum_{J\in\A(\I)\setminus\{I\} }|O_J| = \sum_{J\in\A(\I)\setminus\{I\} } p_J \cdot |O_J|\\
    &= \sum_{J\in\A(\I)\setminus\{I\} }|\{ (\bm{\delta},\bm{\varepsilon}) \ | \ \bm{\delta}, \bm{\varepsilon} \in B; \ (\bm{\delta},\bm{\varepsilon}) \in O_J \}| = k(k-1).
	\end{align*}
Thus the constant must be $c = k(k-1)/v(v-1)$. Hence $\D = (\PP,\B)$ is a $2$-design  if and only if  
    \begin{equation}\label{eq:ancestralcondition}
    |(B\times B) \cap O_J| = \frac{ k(k-1)}{v(v-1)}|O_J| \ \ \text{ for each proper ancestral subset }J.
    \end{equation}

Our next task is to determine, for each proper ancestral subset  $J$, the cardinality
\[
|(B\times B) \cap O_J|
    = |\{ (\bm{\delta},\bm{\varepsilon}) \ | \ \bm{\delta},\bm{\varepsilon} \in B; \ (\bm{\delta},\bm{\varepsilon}) \in O_J \}|.
\]

Consider an arbitrary element $\bm{\delta}\in\PP$.
The number of  pairs in $(B\times B) \cap O_J$ with first entry $\bm{\delta}$ is 
$|\{ \bm{\varepsilon} \in B \ | \ (\bm{\delta},\bm{\varepsilon}) \in O_J \}|$ if $\bm{\delta}\in B$, and there are no such pairs if $\bm{\delta}\notin B$. 
To evaluate the size of this subset, assume that $\bm{\delta}\in B$. Then the elements $\bm{\varepsilon}$ contributing to pairs $(\bm{\delta},\bm{\varepsilon})$ in $(B\times B) \cap O_J$ are, by \eqref{O_J}, those elements such that (i) $\bm{\varepsilon}\pi_J=\bm{\delta}\pi_J$, and (ii) for each $i\in \Bdy{J}$ (so $J\cup\{i\}$ is ancestral),  $\bm{\varepsilon} \pi_{J \cup \{i\}} \ne \bm{\delta} \pi_{J \cup \{i\}}$. Recalling the definition  of the partition $\C_J$ in \eqref{partition}, condition (i) is equivalent to requiring $\bm{\varepsilon}, \bm{\delta}$ to lie in the same part $C_{\bm{\nu}}$  of  $\C_J$ as defined in \eqref{partition}, where  ${\bm{\nu}} = \bm{\delta}\pi_J\in \Delta_J$.  Moreover,  by the definition of the array function in 
\eqref{arrayfn},  $|B\cap C_{\bm{\nu}}| = ({\bm{\nu}})\chi_B$. Similarly, condition (ii) requires us to exclude the union of the sets $B\cap C_{\bm{\delta}\pi_{J\cup \{i\}}}$ for the parts $C_{\bm{\delta}\pi_{J\cup \{i\}}}$ of $\C_{J\cup\{i\}}$, with $i\in\Bdy{J}$, and by \eqref{arrayfn}, $|B\cap C_{\bm{\delta}\pi_{J\cup \{i\}}}| = (\bm{\delta}\pi_{J\cup \{i\}})\chi_B$.
Thus, noting that $J\cup S$ is an ancestral subset for each $S\subseteq \Bdy{J}$,
    \begin{align*}
    |(B\times B) \cap O_J|
    &= \sum_{\bm{\delta} \in B}  |\{ \bm{\varepsilon} \in B \ | \ \bm{\delta},\bm{\varepsilon}) \in O_J \}| \\
    &= \sum_{\bm{\delta} \in B}  \Big| \{ \bm{\varepsilon} \in B \ | \ \bm{\varepsilon} \pi_J = \bm{\delta} \pi_J \} \setminus \bigcup_{i \in \Bdy{J}} \{ \bm{\varepsilon} \in B \ | \ \bm{\varepsilon} \pi_{J \cup \{i\}} = \bm{\delta} \pi_{J \cup \{i\}} \} \Big| \\
    &= \sum_{\bm{\delta} \in B} \Big( \sum_{S \subseteq \Bdy{J}} (-1)^{|S|} (\bm{\delta}\pi_{J \cup S})\chi_B \Big)\\
    &= \sum_{S \subseteq \Bdy{J}} (-1)^{|S|} \Big( \sum_{\bm{\delta} \in B} (\bm{\delta}\pi_{J \cup S})\chi_B \Big)\\
    &=\sum_{S \subseteq \Bdy{J}} (-1)^{|S|}  \Big(\sum_{\bm{\delta} \in B} |B \cap C_{\bm{\delta}\pi_{J \cup S}}| \Big)
    \end{align*}
where, for the third equality, we use an analogous argument to that given above to prove \eqref{eq:O_J2}.

Consider $\bm{\nu} \in \Delta_{J \cup S}.$ If $\bm{\nu} \notin B\pi_{J \cup S}$, then  $(\bm{\nu})\chi_B =|B \cap C_{\bm{\nu}}|= 0$. If $\bm{\nu}=\bm{\delta}\pi_{J \cup S} \in  B\pi_{J \cup S}$, then  $(\bm{\nu})\chi_B=|B \cap C_{\bm{\nu}}|=|B \cap C_{\bm{\delta}\pi_{J \cup S}}|$, and there are $(\bm{\nu})\chi_B$ distinct $\bm{\delta'}\in B$ such that $\bm{\delta'}\pi_{J \cup S} = \bm{\nu}$. Thus
    \begin{equation}\label{arraysums}
   \sum_{\bm{\delta} \in B} |B \cap C_{\bm{\delta}\pi_{J \cup S}}|
    = \sum_{\bm{\nu} \in \Delta_{J \cup S}} \left((\bm{\nu})\chi_B\right)^2=\Sq[J\cup S].  
    \end{equation}
It follows that 
    \[|(B\times B) \cap O_J|=  \sum_{S \subseteq \Bdy{J}} (-1)^{|S|}\Sq[J\cup S].\]
Thus, by \eqref{eq:ancestralcondition}, $\D = (\PP,\B)$ is a $2$-design  if and only if  
    \[
    \sum_{S \subseteq \Bdy{J}} (-1)^{|S|}\Sq[J\cup S] = \frac{ k(k-1)}{v(v-1)}|O_J| \ \ \text{ for each proper ancestral subset }J.
    \]
Recalling  the expression for $|O_J|$ in Proposition~\ref{p:orbit}, and the fact that $|\Delta_i| = e_i$ for each $i$ and $v = |\PP| = \prod_{i\in I}|\Delta_i|$, this condition for $J$ becomes
	\begin{align*}
	 \sum_{S \subseteq \Bdy{J}} (-1)^{|S|}\Sq[J\cup S]
          &=  \frac{ k(k-1)}{(v-1)}{ \left( \prod_{i \in \Bdy{J}} (e_i - 1) \right) \left( \prod_{i \in (J \cup \Bdy{J})^\c} e_i \right)} \\
		\end{align*}
This completes the proof of Proposition~\ref{p:2des}.
\end{proof}

We now deduce Theorem \ref{t:gwp-design}.

\begin{corollary} \label{c:gwp-design}
Suppose that Hypothesis~\ref{h:2des} holds. Then $\D = (\PP,\B)$ is a $2$-design if and only  condition \eqref{2des} holds for each proper non-empty ancestral subset $J$ of $I$. Thus the assertion of Theorem \ref{t:gwp-design} is proved.
\end{corollary}

\begin{proof}
By Proposition~\ref{p:2des}, if $\D = (\PP,\B)$ is a $2$-design then \eqref{2deswithmu} holds for each proper ancestral subset $J$, or equivalently, \eqref{2des}  holds for each proper ancestral subset $J$ (using \eqref{eq:sum-squares}). 

Conversely suppose that  \eqref{2des} holds for each proper non-empty ancestral subset $J$. 
In the proof of Proposition~\ref{p:2des}, it was shown  that  $\D = (\PP,\B)$ is a $2$-design if and only if the equation in  \eqref{eq:ancestralcondition} holds for each proper ancestral subset $J$ of $I$, and by \eqref{eq:sum-squares} this equation for such a subset $J$ is equivalent to  condition \eqref{2des} for $J$.
Thus by our assumption, the equation in  \eqref{eq:ancestralcondition} holds for  all proper non-empty ancestral subsets $J$, and it remains to prove that it holds for the subset $J=\varnothing$. 

By \cite[Theorem C]{BPRS}, the union of the $O_J$ over all $J \in\A(\I)\setminus\{I\}$ (the proper ancestral subsets) is equal to the set of all ordered pairs of distinct points from $\PP$, a set of size $v(v-1)$. Hence the sum of the quantity $|(B\times B)\cap O_J|$ in \eqref{eq:ancestralcondition} over all subsets $J \in\A(\I)\setminus\{I\}$ is equal to  $k(k-1)$. Thus the first equality below holds, and using the equation in \eqref{eq:ancestralcondition} for all $J \in\A(\I)\setminus\{I,\varnothing\}$, we obtain 
 \begin{align*}
    |(B\times B) \cap O_{\varnothing}|
    &= k(k-1) - \sum_{J \in \A(\I), J\ne I, \varnothing} |(B\times B) \cap O_J| \\
    &= k(k-1) - \frac{k(k-1)}{v(v-1)} \sum_{J \in \A(\I), J\ne I, \varnothing} |O_J| \\
    &= k(k-1) - \frac{k(k-1)}{v(v-1)} \left( v(v-1) - |O_{\varnothing}| \right)\\
    &= \frac{k(k-1)}{v(v-1)} |O_{\varnothing}|.
    \end{align*}
Therefore the equation in \eqref{eq:ancestralcondition} holds also for $J=\varnothing$, and hence the equation in \eqref{eq:ancestralcondition} holds  for all $J \in\A(\I)\setminus\{I\}$, so  $\D = (\PP,\B)$ is a $2$-design.
Thus the proof is complete.
\end{proof}

\subsection{Application to chains and antichains}

Here we show that our main result Theorem~\ref{t:gwp-design} specialises to earlier results in the literature for the special cases where $\I$ is a chain (that is, $\I$ is a totally ordered set) and an anti-chain (that is, the relation $\prec$ is empty). See
\cite[Theorem 1.3]{chainspaper}  and \cite[Theorem 1.3]{multigrids} for these special cases.

\begin{lemma} \label{l:2des-antich}
If $(I,\preccurlyeq)$ is an $s$-antichain for some $s \geq 2$, then Theorem \ref{t:gwp-design} is equivalent to \cite[Theorem 1.3]{multigrids}.
\end{lemma}

\begin{proof}
It is enough to show that the set of conditions \eqref{2des}, for all proper nonempty ancestral subsets $J$ of $I$, is equivalent to the set of conditions (6) in \cite[Theorem 1.3]{multigrids}, for all proper nonempty ancestral subsets $J$ of $I$. In \cite{multigrids} the set $\Delta_{J \cup S}$ is denoted as $\mathscr{E}_{J \cup S}$, so using our current notation, \cite[Theorem 1.3, conditions (6)]{multigrids} can be written as 
    \begin{equation}
    \sum_{\bm{\nu} \in \Delta_J} \big((\bm{\nu})\chi_B\big)^2 = k + \frac{k(k-1)}{v-1} \left( \bigg( \prod_{i \in J^\c} e_i \bigg) - 1 \right) \quad \text{for all $J \subsetneq I$ with $J \ne \varnothing$.}
    \end{equation}    
Since $(I,\preccurlyeq)$ is an antichain, any subset $J$ of $I$ is ancestral, with border $\Bdy{J} = J^\c$. Thus, for any subset $J$ of $I$, we have $(J \cup \Bdy{J})^\c = (J \cup J^\c)^\c = I^\c = \varnothing$ and we can write \eqref{2des} as
    \begin{equation} \label{2des-antich}
    \sum_{S \subseteq J^\c} (-1)^{|S|} \left(\sum_{\bm{\nu} \in \Delta_{J \cup S}} (\bm{\nu})\chi_B\right)^2 = \frac{k(k-1)}{v-1} \left( \prod_{i \in J^\c} (e_i - 1) \right).
    \end{equation}
It follows from Theorem \ref{t:gwp-design} and Proposition \ref{p:2des} that the set of conditions \eqref{2des-antich}, for all proper nonempty subsets $J$, is equivalent to the set of conditions \eqref{2des-antich} for all proper subsets of $J$, including $\varnothing$.
So the set of conditions \eqref{2des-antich}, for all proper subsets $J$, is equivalent to condition (21) in \cite{multigrids}. It is shown in the proof of \cite[Theorem 1.3]{multigrids} that \cite[condition (21)]{multigrids} is equivalent to \cite[condition (6)]{multigrids}. This proves that, if $(I,\preccurlyeq)$ is an antichain, Theorem \ref{t:gwp-design} is equivalent to \cite[Theorem 1.3]{multigrids}.
\end{proof}

\begin{lemma} \label{l:2des-ch}
If $(I,\preccurlyeq)$ is an $s$-chain, where $s \geq 2$, then Theorem \ref{t:gwp-design} is equivalent to \cite[Theorem 1.3]{chainspaper}.
\end{lemma}

\begin{proof}
As in Lemma \ref{l:2des-antich}, we will show that the set of conditions \eqref{2des}, for all proper nonempty ancestral subsets $J$ of $I$, is equivalent to the set of conditions (2) and (3) in \cite[Theorem 1.3]{chainspaper}.

Suppose that $I = \{1, \ldots, s\}$ with $1 \prec 2 \prec \ldots \prec s$. Then the proper ancestral subsets $J$ of $I$ are $\varnothing$ and the sets $A[i]$ for $2 \leq i \leq s$.  We first recall the conditions (2) and (3) in \cite[Theorem 1.3]{chainspaper}. For any $i \in \{1, \ldots, s-1\}$, the partition $\C_i$ in \cite{chainspaper} is, in our current notation, the partition $\C_{A[i+1]}$. Thus any class $C \in \C_i$ corresponds to a class $C_{\bm{\nu}}$ for some $\bm{\nu} \in \Delta_{A[i+1]}$, and the parameter $x_C$ in \cite{chainspaper} is equal to $(\bm{\nu})\chi_B$. As the class $C$ varies over the partition $\C_i$, the tuple $\bm{\nu}$ varies over $\Delta_{A[i+1]}$. Thus, when $i = 1$, condition (2) of \cite[Theorem 1.3]{chainspaper} can be written in our current notation as
    \begin{equation} \label{2des-chains-1}
    \sum_{\bm{\nu} \in \Delta_{A[2]}} (\bm{\nu})\chi_B \left( (\bm{\nu})\chi_B - 1 \right)
    = \frac{k(k-1)}{v-1} (e_1 - 1)
    \end{equation}
The class $C^+$ in \cite{chainspaper} is the unique $\C_{i-1}$-class that contains $C$, so $C^+ = C_{\bm{\nu}\pi_{A[i]}}$, the unique $\C_{A[i]}$-class that contains $C$. Hence conditions (3) of \cite[Theorem 1.3]{chainspaper} can be written as
    \begin{equation} \label{2des-chains-i>1}
    \sum_{\bm{\nu} \in \Delta_{A[i]}} (\bm{\nu})\chi_B \left( (\bm{\nu}\pi_J)\chi_B - (\bm{\nu})\chi_B \right) = \frac{k(k-1)}{v-1} (e_i - 1) \prod_{j \leq i-1} e_j \quad \text{for $i \in \{2, \ldots, s\}$}.
    \end{equation}

We now prove that the set of conditions \eqref{2des}, for all proper nonempty ancestral subsets $J$, is equivalent to conditions \eqref{2des-chains-1} and \eqref{2des-chains-i>1}. Note that for each proper ancestral subset $J$ the border $\Bdy{J}$ is a singleton set, and in particular,
    \begin{equation} \label{border-ch}
    \Bdy{J} = \begin{cases} \{i-1\} &\text{if $J = A[i]$, $2 \leq i \leq s$} \\
    \{s\} &\text{if $J = \varnothing$} \end{cases}
    \end{equation}
For convenience let
    \[
    J^* := J \cup \Bdy{J}.
    \]

It follows from Theorems \ref{t:gwp-design} and Proposition \ref{p:2des} that the set of conditions \eqref{2des}, for all proper nonempty ancestral subsets $J$ of $I$, is equivalent to the set of conditions \eqref{2des}, for all proper ancestral subsets $J$ of $I$ (that is, including $J = \varnothing$). So the left side of \eqref{2des}, which is equivalent to \eqref{2deswithmu}, is $\Sq[J] - \Sq[J^*]$. Using \eqref{arraysums} we can write this as
    \begin{equation} \label{arraysums2}
    \Sq[J] - \Sq[J^*]
    = \sum_{\bm{\delta} \in B}  \left( (\bm{\delta}\pi_J)\chi_B - (\bm{\delta}\pi_{J^*})\chi_B \right).
    \end{equation}

If $J = A[2]$ then $J^* = A[1] = I$, and \eqref{arraysums2} can be written as
    \[
    \Sq[{A[2]}] - \Sq[I]
    = \sum_{\bm{\delta} \in B}  \left( (\bm{\delta}\pi_{A[2]})\chi_B - (\bm{\delta})\chi_B \right)
    = \sum_{\bm{\delta} \in B} \left( (\bm{\delta}\pi_{A[2]})\chi_B - 1 \right).
    \]
 The number of times the term $(\bm{\delta}\pi_{A[2]})\chi_B - 1$ appears in the sum on the right is equal to the number of points $\bm{\delta'} \in B$ satisfying $\bm{\delta'}\pi_{A[2]} = \bm{\delta}\pi_{A[2]}$. That is, the number of times it appears is equal to $|B \cap C_{(\bm{\delta}\pi_{A[2]})}| = (\bm{\delta}\pi_{A[2]})\chi_B$. Since $\bm{\delta}\pi_{A[2]}$ varies over $\Delta_{A[2]}$ as $\bm{\delta}$ varies over $\PP$,
    \[
    \sum_{\bm{\delta} \in B}  \left( (\bm{\delta}\pi_{A[2]})\chi_B - 1 \right)
    = \sum_{\bm{\nu} \in \Delta_{A[2]}} (\bm{\nu})\chi_B \left( (\bm{\nu})\chi_B - 1 \right).
    \]
So condition \eqref{2des} for $J = A[2]$ is equivalent to condition \eqref{2des-chains-1}.

Suppose now that $J = A[i]$, $3 \leq i \leq s$, or $J = \varnothing$. Note that the number of times the term $(\bm{\delta}\pi_J)\chi_B - (\bm{\delta}\pi_{J^*})\chi_B$ appears in the sum on the right of \eqref{arraysums2} is equal to the number of points $\bm{\delta'} \in B$ satisfying $\bm{\delta'}\pi_{J^*} = \bm{\delta}\pi_{J^*}$. That is, the number of times it appears is equal to the number of points in $B \cap C_{\bm{\delta}\pi_{J^*}}$, which is equal to $(\bm{\delta}\pi_{J^*})\chi_B$. Note also that, since $J \subseteq J^*$, we have $\pi_J = \pi_{J^*}\pi_J$ by \eqref{eq:proj}. As $\bm{\delta}$ varies over $\PP$, the projection $\bm{\delta}\pi_{J^*}$ varies over $\Delta_{J^*}$, so that \eqref{arraysums2} can be written as
    \[
       \Sq[J] - \Sq[J^*]
    = \sum_{\bm{\nu} \in \Delta_{J^*}} (\bm{\nu})\chi_B \left( (\bm{\nu}\pi_J)\chi_B - (\bm{\nu})\chi_B \right).
    \]
If $J = A[i]$, $3 \leq i \leq s$, then $J^* = A[i] \cup \{i-1\} = A[i-1]$, and if $J = \varnothing$ then $J^* = \{s\} = A[s]$. Hence
    \begin{align*}
    \Sq[J] - \Sq[J^*]
    &=  \begin{cases}
        \sum_{\bm{\nu} \in \Delta_{A[i-1]}} (\bm{\nu})\chi_B \left( (\bm{\nu}\pi_J)\chi_B - (\bm{\nu})\chi_B \right) &\text{if $J = A[i]$, $3 \leq i \leq s$} \\
        \sum_{\bm{\nu} \in \Delta_{A[s]}} (\bm{\nu})\chi_B \left( (\bm{\nu}\pi_J)\chi_B - (\bm{\nu})\chi_B \right) &\text{if $J = \varnothing$}
        \end{cases} \\
    &=  \begin{cases}
        \sum_{\bm{\nu} \in \Delta_{A[i]}} (\bm{\nu})\chi_B \left( (\bm{\nu}\pi_J)\chi_B - (\bm{\nu})\chi_B \right) &\text{if $J = A[i]$, $2 \leq i \leq s-1$} \\
        \sum_{\bm{\nu} \in \Delta_{A[s]}} (\bm{\nu})\chi_B \left( (\bm{\nu}\pi_J)\chi_B - (\bm{\nu})\chi_B \right) &\text{if $J = \varnothing$}.
        \end{cases}
    \end{align*}
Substituting each of these expressions into the left side of \eqref{2des}, we see that condition \eqref{2des} for $J = A[i]$, $3 \leq i \leq s$, is equivalent to condition \eqref{2des-chains-i>1} for $i \in \{2, \ldots, s-1\}$, and condition \eqref{2des} for $J = \varnothing$ is equivalent to condition \eqref{2des-chains-i>1} for $i = s$.

This shows that the set of conditions \eqref{2des}, for all proper ancestral subsets $J$, is equivalent to the set of conditions (2) and (3) in \cite[Theorem 1.3]{chainspaper}.
\end{proof}

\section{Families of examples  for posets of small order} \label{sec:ex}

We illustrate the concepts in Section \ref{s:prelims} for some small posets $(I,\preccurlyeq)$. Examples \ref{poset:grid} and \ref{poset:2ch} correspond to \cite[structures (2) and (3)]{N} and to \cite[Figure 1, diagrams (3) and (4)]{Thr}. Figure \ref{fig:posets-4} is listed in \cite[Figure 3, diagram (7)]{Thr} and is described in \cite[Example 1 and Figure 3]{BPRS}.

The partitions $\CC^*$ corresponding to small posets $(I,\preccurlyeq)$ with $|I| \leq  3$ are shown in Figures \ref{fig:posets-2}, \ref{fig:posets-3},   and one example with $|I|=4$  is shown in Figure \ref{fig:posets-4}.  We reproduce these below with more details for those examples that are not chains or antichains. We label the partitions $\C_{A[i]}$ for each $i \in I$; all other invariant partitions can be obtained by taking intersections of classes from two or more partitions, as discussed in Section~\ref{s:posetbkstr}. For example, in Figure \ref{fig:3grid}, each class in the partition $\C_{\{1,2\}}$ is the intersection of a $\C_{\{1\}}$-class and a $\C_{\{2\}}$-class; since each $\C_{\{1\}}$- and $\C_{\{2\}}$-class is a vertical plane, each $\C_{\{1,2\}}$-class is a vertical line. The groups $\prod_{(I,\preccurlyeq)}(G_i,\Delta_i)$ for Figures \ref{fig:posets-2} and \ref{fig:posets-3} are obtained by applying Theorems \ref{thm:dir} and \ref{thm:wr}, as appropriate. 

Examples of $2$-designs whose point set admits a poset of partitions as shown in Figures \ref{fig:posets-2}, \ref{fig:3ch}, and \ref{fig:3grid} can be found in \cite{chainspaper} (for Figures \ref{fig:2ch} and \ref{fig:3ch}) and \cite{multigrids} (for Figures \ref{fig:2grid} and \ref{fig:3grid}), as described in Section \ref{sec:chains}. In this section we thus give examples for the remaining posets in Figure \ref{fig:posets-3} (that is, all such posets that are not chains nor antichains), and for the poset in Figure \ref{fig:posets-4}.

\subsection{Posets $(I,\preccurlyeq)$ with $|I| = 3$} \label{ss:I=3}

The posets considered here are those in Figure~\ref{fig:posets-3}(c--e).
In each of the constructions below $p \geq 2$ is an integer and
    \begin{equation} \label{eq:e-3}
    e_1 = p^2 + p + 1, \ e_2 = p^2 - p + 1, \ e_3 = p^4 - p^2 + 1.
    \end{equation}
   Now $|\Delta_i|=e_i$ and it is convenient to take $\Delta_i=\{0,1,2,\ldots, e_i-1\}\subset\mathbb{Z}$. Note that $e_2<e_1<e_3$ and, as integer subsets we have $\Delta_2\subset \Delta_1\subset \Delta_3$. 
Thus $\PP = \Delta_1 \times \Delta_2 \times \Delta_3$, so that
    \[ v = |\PP| = e_1e_2e_3 = q^2 + q + 1, \quad \text{where} \ q = p^4. \]
Furthermore, each block will have size $k = q + 1$, so that 
    \[ \frac{k(k-1)}{v-1} = \frac{(q+1)q}{q^2 + q} = 1 \]
and condition \eqref{2deswithmu},  for a proper non-empty ancestral subset $J$, becomes
    \begin{equation} \label{2des-examples}
    \sum_{S \subseteq \Bdy{J}} (-1)^{|S|} \Sq[J\cup S]
    = \left( \prod_{i \in \Bdy{J}} (e_i - 1) \right) \left( \prod_{j \in (J \cup \Bdy{J})^\c} e_j \right)= \frac{|O_J|}{v}
     \end{equation}
where the last equality follows from \eqref{|O_J|}. In each construction the block $B$ will be described as a disjoint union of subsets $B_1$, $B_2$, and $B_3$. Note that in this notation the subsets $B_1$, $B_2$, and $B_3$ have no association with the nodes $1$, $2$, and $3$ of the poset $(I,\preccurlyeq)$.

\subsubsection{Example for Figure \ref{fig:ch-grid}}

Suppose that the point set $\PP$ admits a poset of partitions that corresponds to the disconnected poset $(I,\preccurlyeq)$ with nodes $1$, $2$, and $3$, as shown in Figure \ref{fig:ch-grid}, which we show again below. 

\begin{center}
    \includegraphics{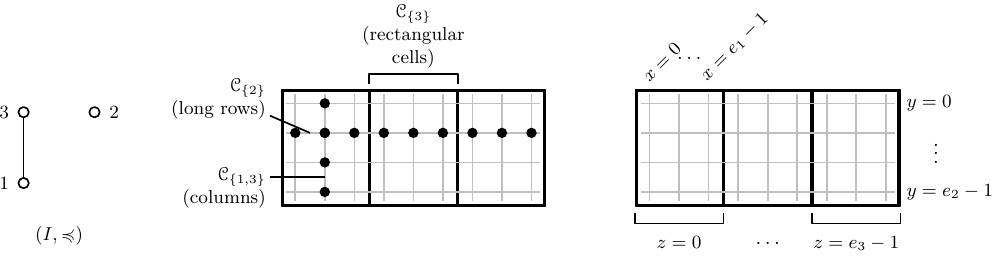}
\end{center}

Let $G = (G_1 \wr G_3) \times G_2$, where $G_i$ is a $2$-transitive subgroup of $S_{e_i}$, for each $i$.  It follows from Theorems~\ref{thm:dir} and~\ref{thm:wr} that $G$ is a generalised wreath product group preserving this poset of partitions. In particular, these results imply that  the proper nonempty ancestral subsets of $I$ are $\{2\}$, $\{3\}$, $\{1,3\}$, and $\{2,3\}$, so the nontrivial point partitions of $\PP$ preserved by $G$ are:
    \begin{itemize}
    \item $\C_{\{2\}}$ with classes $C_y=\{(x,y,z) \mid x\in\Delta_1, \ z\in\Delta_3\}$, for $y \in \Delta_2$; the $\C_{\{2\}}$-classes correspond to the long rows in Figure~\ref{fig:ch-grid};
    \item $\C_{\{3\}}$ with classes $C_z=\{ (x,y,z) \mid x\in\Delta_1, \ y\in\Delta_3\}$, for $z\in \Delta_3$; the $\C_{\{3\}}$-classes correspond to the rectangular cells in Figure~\ref{fig:ch-grid};
    \item $\C_{\{1,3\}}$ with classes $C_{(x,z)}=\{(x,y,z) \mid y\in\Delta_2\}$ for $(x,z)\in \Delta_{\{1,3\}}=\Delta_1\times \Delta_3$; the $\C_{\{1,3\}}$-class $C_{(x,z)}$ corresponds to the small $x$-column in the rectangular cell $C_z\in\C_{\{3\}}$ in Figure~\ref{fig:ch-grid};
    \item $\C_{\{2,3\}}$ with classes $C_{(y,z)}=\{(x,y,z) \mid x\in\Delta_1\}$ for $(y,z)\in \Delta_{\{2,3\}}=\Delta_2\times \Delta_3$; the $\C_{\{2,3\}}$-class $C_{(y,z)}$ corresponds to the part of the long row $C_y$ in the rectangular cell $C_z$ in Figure~\ref{fig:ch-grid}.
    \end{itemize}

Thus for a point $\bm{\delta}=(x,y,z)\in\PP$, the third coordinate represents the rectangular cell $C_{z}\in\C_{\{3\}}$ containing $\bm{\delta}$, the second coordinate represents the long row $C_{y}\in\C_{\{2\}}$ containing $\bm{\delta}$, and the first coordinate represents the small column $C_{(x,z)}$ inside $C_{z}$ containing $\bm{\delta}$.

 Note that the special ancestral subsets $A[j]$ defined in Subsection~\ref{s:posets} are as follows: $A[1]=\{1,3\},A[2]=\{2\},A[3]=\{3\}$, so the sub-poset  $ \CC^* = \{ \C_{A[i]} \ | \ i \in I \}$ in \eqref{C*} contains partitions for three of the four proper non-empty ancestral sets. The remaining ancestral set $\{2,3\}$ has the property that 
every $\C_{\{2,3\}}$-class is the intersection of a $\C_{\{2\}}$-class and a $\C_{\{3\}}$-class, as explained in Section \ref{s:posetbkstr}.

\begin{example} \label{ex:ch-grid}
Take $B ={\color{amber} B_1} \cup {\color{red} B_2} \cup {\color{blue} B_3}$ where ${\color{red}B_2} = {\color{red}B_{2,1}} \cup {\color{red}B_{2,2}}$ and, recalling that $\Delta_2\subset \Delta_1\subset \Delta_3\subset \mathbb{Z}$, and that $e_1=p^2+p+1, e_2=p^2-p+1, e_3=p^4-p^2+1$ with $q=p^4$:
    \begin{align*}
    {\color{amber} B_1} &:= \{ (x,0,0) \ | \ \ 0 \leq x \leq p \} , \\
    {\color{red} B_{2,1}} &:= \{ (p+1,\, y, \,0) \ | \ 1 \leq y \leq p \}, \\
    {\color{red} B_{2,2}} &:= \{ (x,\, x - 1,\, 0) \ | \ p+2 \leq x \leq e_2\}, \\
    {\color{blue} B_3} &:= \{ (0,y,z) \ | \ 1\leq y\leq e_2-1, \ (y - 1)(p^2 + p) + 1 \leq z \leq y (p^2 + p) \}.
    \end{align*}

Note that in {\color{blue} $B_3$}, $z$ takes all values from $(1 - 1)(p^2 + p) + 1=1 $ to $(e_2-1) (p^2 + p)=p^4-p^2=e_3-1$, that is, all the values in $\Delta_3\setminus\{0\}$, and each such $z$ corresponds to a unique value of $y$. In ${\color{red} B_{2,2}}$, the first coordinate $x\leq e_2<e_1$ so $x\in \Delta_1$, and the second coordinate $x-1\leq e_2-1$ so $x-1\in \Delta_2$.
    
Thus $|{\color{amber} B_1}| = p+1$, $|{\color{red} B_2}| = p + (e_2 - p - 1) = e_2 - 1 = p^2 - p$, and $|{\color{blue} B_3}| = (e_2 - 1)(p^2 + p) = (p^2 - p)(p^2 + p) = p^4 - p^2$. It follows that
    \[ |B| = (p+1) + (p^2 - p) + (p^4 - p^2) = p^4 + 1 = q+1. \]
The distribution of the points of $B$ among the partition classes is illustrated in Figure \ref{fig:ex-ch-grid}, where classes are numbered from left to right and top to bottom. For instance long row $C_0\in\C_{\{2\}}$ is the top row and rectangular cell $C_0\in\C_{\{3\}}$ is the leftmost rectangle.

\begin{figure}[ht]
    \centering
    \includegraphics[width=\textwidth]{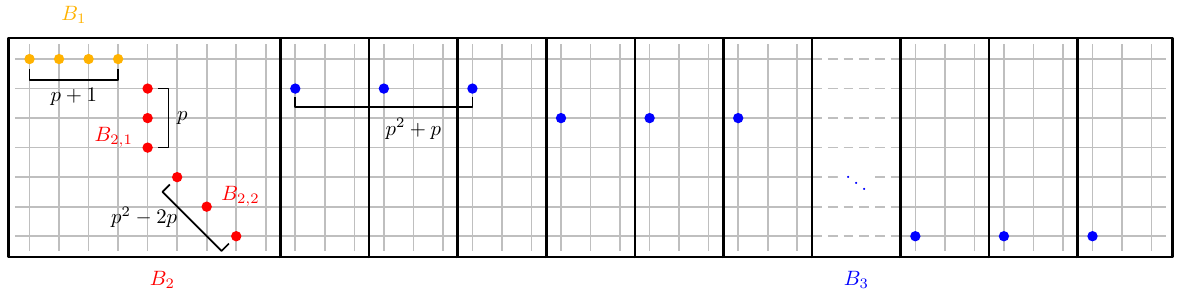}
    \caption{Distribution of points of $B$ in Example \ref{ex:ch-grid}}
    \label{fig:ex-ch-grid}
\end{figure}

The array of $B$ consists of the parameters $(\bm{\nu})\chi_B$ for each $\bm{\nu} \in \Delta_J$ and each proper ancestral subset $J$. By Theorem \ref{t:gwp-design}, the structure $\D = (\PP,B^G)$ is a $2$-design if and only if condition \eqref{2des-examples} holds for any proper nonempty ancestral subset $J$ of $I$. As we mentioned above, the proper nonempty ancestral subsets $J$ of $I$ are $\{2\}$, $\{3\}$, $\{2,3\}$ and $\{1,3\}$.
    
Table \ref{tab:ex-chgrid} lists, for each $J$, the possible non-zero values $(\bm{\nu})\chi_B$ (for ${\bm \nu}\in \Delta_J$), together with the number $\# C_{\bm{\nu}}$ of classes of $\C_J$ giving each such value $(\bm{\nu})\chi_B$. Following this, the value of $\Sq[J]$
is computed by adding, over these values,  the product of $\#C_{\bm{\nu}}$ and $\big((\bm{\nu})\chi_B\big)^2$. The result is listed in the last row of Table \ref{tab:ex-chgrid}.

For instance, for $J = \{2\}$, the class $C_{0} \in \C_{\{2\}}$ contains all $p+1$ points in $B_1$, so that $(0)\chi_B = p+1$,while each of the remaining $p^2 - p$ classes in $\C_{\{2\}}$ contains one point in $B_2$ and $p^2 + p$ points in $B_3$, so that $(\bm{\nu})\chi_B = p^2 + p + 1$ for each $\bm{\nu} \in \{1, \ldots, p^2 - p\}$. This gives the values in the first two rows of column $\{2\}$ of Table \ref{tab:ex-chgrid}, and thus
    \begin{align*}
   \Sq[\{2\}]= \sum_{\bm{\nu} \in \Delta_2} \big((\bm{\nu})\chi_B\big)^2
    &= 1 \cdot (p+1)^2 + (p^2 - p)(p^2 + p + 1)^2 
    \end{align*} 
which is equal to the entry in the   $\Sq[J]$-row of column $\{2\}$. This verifies that all entries in column $\{2\}$ of Table~\ref{tab:ex-chgrid} are correct. The other entries of Table~\ref{tab:ex-chgrid} are computed similarly.

\begin{table}[ht]
     \centering
     \begin{tabular}{r||r|r||r|r||r|r||r|r}
     \hline
     \multirow{2}{1.5cm}{\parbox[r]{1.5cm}{\raggedleft $J$ s.t. \\ $\bm{\nu} \in \Delta_J$}} & \multicolumn{2}{c||}{${\{2\}}$} & \multicolumn{2}{c||}{${\{3\}}$} & \multicolumn{2}{c||}{${\{2,3\}}$} &\multicolumn{2}{c}{${\{1,3\}}$} \\
     \cline{2-9}
     & $\# C_{\bm{\nu}}$ & $(\bm{\nu})\chi_B$ & $\# C_{\bm{\nu}}$ & $(\bm{\nu})\chi_B$ & $\# C_{\bm{\nu}}$ & $(\bm{\nu})\chi_B$ & $\# C_{\bm{\nu}}$ & $(\bm{\nu})\chi_B$\\
     \hline\hline
     & $1$         & $p+1$ & $1$                    & $p^2+1$         & $1$         & $p + 1$&$1$&$p$ \\
     & $p^2 - p$   & $p^2+p+1$   &  $p^4 - p^2$ & $ 1$ & $p^4 - p$ & $1$ &$p^4-p+1$&$1$\\
     \hline
     $\Sq[J]$\phantom{\Big|} & \multicolumn{2}{c||}{\parbox[c]{3cm}{\vspace{3pt} $p^6+p^5+p^4-p^3$ \\ $+ p + 1$}} & \multicolumn{2}{c||}{$2p^4 + p^2 + 1$} & \multicolumn{2}{c||}{$p^4 + p^2 + p + 1 $}& \multicolumn{2}{c}{$p^4 + p^2-p + 1$} \\
     \hline
     \end{tabular}
     \caption{Values of $\Sq[J]$ for $J =  \{2\}, \{3\},\{2,3\},\{1,3\}$}
     \label{tab:ex-chgrid}
 \end{table}

Note that the right side of \eqref{2des-examples} is $|O_J|/v$ and this value together with the boundary $\Bdy{J}$ is given, for each $J$, in Table \ref{tab:chgrid-orb}. Recall that the $|O_J|$ are the $G$-orbit lengths on ordered pairs of distinct points, see Section~\ref{s:orbs}.  

\begin{table}[ht]
    \centering
    \begin{tabular}{lll}
    \hline
    $J$ & $\Bdy{J}$ & $|O_J|/v$ \\
    \hline  
    $\{2\}$ & $\{3\}$  & $(e_3 - 1)e_1=(p^4-p^2)(p^2+p+1)$ \\
    $\{3\}$ & $\{1,2\}$  & $(e_1 - 1)(e_2-1)=p^4-p^2$ \\
    $\{2,3\}$ & $\{1\}$  & $(e_1 - 1)=p^2+p$ \\
    $\{1,3\}$ & $\{2\}$  & $(e_2 - 1)=p^2-p$ \\   
      \hline
    \end{tabular}
    \caption{Orbit sizes of $G = (G_1 \wr G_3) \times G_2$ on ordered pairs of distinct points}
    \label{tab:chgrid-orb}
\end{table}

For each proper nonempty ancestral subset $J$, we now compute the left side of \eqref{2des-examples},  and show, on checking Table~\ref{tab:chgrid-orb}, that  \eqref{2des-examples} holds for each proper non-empty ancestral set $J$. Recall from Lemma~\ref{lem:array-sqs} that $\Sq[I]=k=p^4+1$.

If $J = \{2\}$ then  $\Bdy{J}=\{3\}$, so 
\begin{align*} 
    \sum_{S \subseteq \Bdy{J}} (-1)^{|S|} \Sq[J\cup S]
    &=\Sq[\{2\}]-\Sq[\{2,3\}]\\
       &=p^6+p^5+p^4-p^3+p+1-\left( p^4 + p^2 + p + 1\right)\\
     &= p^6+p^5-p^3-p^2\\
               &= \left(p^4 - p^2\right)\left(p^2 + p + 1\right)
    \end{align*}
    
If $J = \{3\}$ then $\Bdy{J}=\{1,2\}$, so 
    \begin{align*} 
    \sum_{S \subseteq \Bdy{J}} (-1)^{|S|} \Sq[J\cup S]
    &= \Sq[\{3\}]-\Sq[\{2,3\}]-\Sq[\{1,3\}]+\Sq[I]\\    
       &=\left(2p^4 + p^2 + 1\right)-\left(p^4+p^2+p+1\right)-\left(p^4+p^2-p+1\right)+p^4+1\\
    &=p^4-p^2
    \end{align*}

If $J = \{2,3\}$, $\Bdy{J}=\{1\}$, so 
    \begin{align*} 
    \sum_{S \subseteq \Bdy{J}} (-1)^{|S|} \Sq[J\cup S]
    &=\Sq[\{2,3\}]-\Sq[I]\\
    &=\left(p^4+p^2+p+1\right)-\left(p^4+1\right)=p^2+p
    \end{align*}

If $J = \{1,3\}$, $\Bdy{J}=\{2\}$, so 
    \begin{align*} 
    \sum_{S \subseteq \Bdy{J}} (-1)^{|S|} \Sq[J\cup S]
    &=\Sq[\{1,3\}]-\Sq[I]\\
    &=\left(p^4+p^2-p+1\right)-\left(p^4+1\right)=p^2-p
    \end{align*} 

Thus \eqref{2des-examples} holds for each proper non-empty ancestral subset $J$ of $I$ and hence, by Theorem~\ref{t:gwp-design}, $\D = (\PP,B^G)$ is a $2$-design.
\end{example}

\subsubsection{Example for Figure \ref{fig:V}}

Suppose that the point set $\PP$ admits a poset of partitions that correspond to the V-shaped poset $(I,\preccurlyeq)$ with nodes $1$, $2$, and $3$, as shown in Figure \ref{fig:V}, which for convenience we show again below. 
    \begin{center}
        \includegraphics{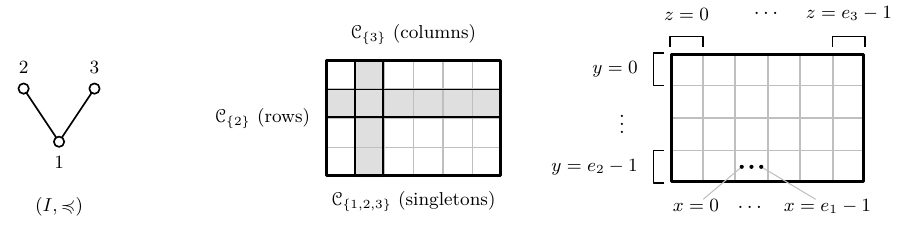}
    \end{center}
Let $G = G_1 \wr (G_2 \times G_3)$, where each $G_i$ is a $2$-transitive subgroup of $S_{e_i}$. 
It follows from Theorems~\ref{thm:dir} and~\ref{thm:wr} that $G$ is a generalised wreath product group preserving this poset of partitions. In particular, these results imply that  the proper nonempty ancestral subsets of $I$ are $\{2\}$, $\{3\}$, and $\{2,3\}$, so the nontrivial point partitions of $\PP$ preserved by $G$ are:
\begin{itemize}
    \item $\C_{\{2\}}$ with classes $C_y=\{(x,y,z) \mid x\in\Delta_1, \ z\in\Delta_3\}$, for $y\in \Delta_2$; the  $\C_{\{2\}}$-classes correspond to the  rows  in Figure~\ref{fig:V};
    \item $\C_{\{3\}}$ with classes $C_z=\{(x,y,z) \mid x\in\Delta_1, \ y\in\Delta_2\}$, for $z\in \Delta_3$; the $\C_{\{3\}}$-classes correspond to the columns in Figure~\ref{fig:V};
    \item $\C_{\{2,3\}}$ with classes $C_{(y,z)}=\{(x,y,z) \mid x\in\Delta_1\}$ for $(y,z)\in \Delta_{\{2,3\}}=\Delta_2\times \Delta_3$; the $\C_{\{2,3\}}$-class $C_{(y,z)}$ is the intersection of row $C_y$ and column $C_z$ in Figure~\ref{fig:V}.
\end{itemize} 

Thus for a point $\bm{\delta}=(x,y,z)\in\PP$, the second coordinate represents the row $C_{y}\in\C_{\{2\}}$ containing $\bm{\delta}$, the third coordinate represents the column $C_{z}\in\C_{\{3\}}$ containing $\bm{\delta}$, and the first coordinate identifies which point of $C_{(y,z)}=C_y\cap C_z$  is equal to $\bm{\delta}$.

Note $\CC^* = \{ \C_{A[i]} \ | \ i \in I \}$ where $A[1]=\{1,2,3\}=I$, $A[2]=\{2\}$, $A[3]=\{3\}$. In this case $\C_{A[1]}=\C_I$ consists of singletons. 

\begin{example} \label{ex:V}
Take $B = {\color{amber} B_1} \cup {\color{red} B_2} \cup {\color{blue} B_3}$ where
    \begin{align*}
    {\color{amber} B_1} &:= \{ (x, 0, 0) \ | \ 0 \leq x \leq p \}, \\
    {\color{red} B_2} &:= \{ (0,y, 0) \ | \ 1 \leq y \leq e_2-1 \}, \ \text{and} \\
       {\color{blue} B_3} &:= \{ (0,y,z) \ | \ 1 \leq y \leq e_2-1, \  (y - 1)(p^2 + p) + 1 \leq z \leq y (p^2 + p) \}     
    \end{align*}
Note that in {\color{blue} $B_3$}, $z$ takes all values from $(1 - 1)(p^2 + p) + 1=1 $ to $(e_2-1) (p^2 + p)=p^4-p^2=e_3-1$, that is, all the values in $\Delta_3\setminus\{0\}$, and each such $z$ corresponds to a unique value of $y$.
     
 Then $|{\color{amber} B_1}| = p+1$, $|{\color{red} B_2}| = e_2 - 1$, and $|{\color{blue} B_3}| = (e_2 - 1)(e_1 - 1) = (p^2 - p)(p^2 + p) = p^4 - p^2 = e_3 - 1$. Hence
      \[ |B| = p + 1 + (e_2 - 1) + (e_3 - 1) = (p + 1)  + (p^2 - p)+ (p^4 - p^2) = p^4 + 1 = q + 1. \]
The distribution of the points in $B$ among the $\C_{\{2,3\}}$-classes, as well as the points that comprise the subsets ${\color{amber} B_1}$, ${\color{red} B_2}$, and ${\color{blue} B_3}$, is shown in Figure \ref{fig:ex-V}, where classes are numbered from left to right and top to bottom.

\begin{figure}[ht]
    \centering
    \includegraphics{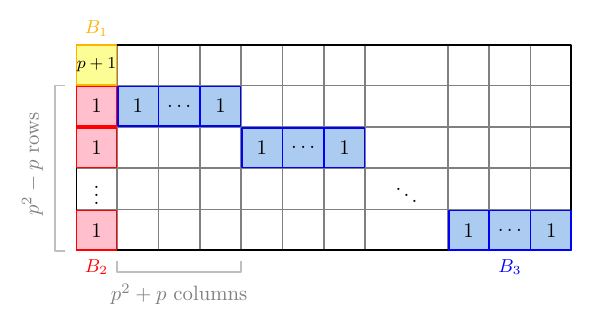}
    \caption{Distribution of points of $B$ in Example \ref{ex:V}}
    \label{fig:ex-V}
\end{figure}

In order to check that the structure $\D = (\PP,B^G)$ is a $2$-design, we will follow similar methods to Example~\ref{ex:ch-grid}:
we need to check that condition \eqref{2des-examples} holds for any proper nonempty ancestral subset $J$ of $I$. As we mentioned above, the  proper nonempty ancestral subsets $J$ of $I$ are $\{1,3\}$, $\{2,3\}$ and $\{3\}$.
    Table \ref{tab:ex-V-inv} lists 
the relevant information to compute $\Sq[J]$ for each such $J$.

For instance, for $J = \{2\}$ the class $C_{0} \in \C_{\{2\}}$ contains all $p+1$ points in $B_1$, so that $(0)\chi_B = p+1$, while each of the remaining $p^2 - p$ classes in $\C_{\{2\}}$ contains one point in $B_2$ and $p^2 + p$ points in $B_3$, so that $(\bm{\nu})\chi_B = p^2 + p + 1$ for each $\bm{\nu} \in \{1, \ldots, p^2 - p\}$. This gives the values in the first two lines of the column $\{2\}$ of Table \ref{tab:ex-V}, and thus
    \begin{align*}
    \Sq[\{2\}]
    &= 1 \cdot (p+1)^2 + (p^2 - p)(p^2 + p + 1)^2 \\
      &= p + (p^4 - p + 1)(p^2 + p + 1)
    \end{align*} 

\begin{table}[ht]
     \centering
     \begin{tabular}{r||r|r||r|r||r|r}
     \hline
     \multirow{2}{1.5cm}{\parbox[r]{1.5cm}{\raggedleft $J$ s.t. \\ $\bm{\nu} \in \Delta_J$}} & \multicolumn{2}{c||}{${\{2,3\}}$} & \multicolumn{2}{c||}{${\{2\}}$} & \multicolumn{2}{c}{${\{3\}}$} \\
     \cline{2-7}
     & $\# C_{\bm{\nu}}$ & $(\bm{\nu})\chi_B$ & $\# C_{\bm{\nu}}$ & $(\bm{\nu})\chi_B$ & $\# C_{\bm{\nu}}$ & $(\bm{\nu})\chi_B$ \\
     \hline\hline
     & $1$         & $p+1$ & $1$                    & $p+1$         & $1$         & $p^2 + 1$ \\
     & $p^4 - p$   & $1$   & \hspace{1cm} $p^2 - p$ & $p^2 + p + 1$ & $p^4 - p^2$ & $1$ \\
     \hline
     $\Sq[J]$\phantom{\Big|} & \multicolumn{2}{c||}{$p^4 + p^2 + p + 1$} & \multicolumn{2}{c||}{$p + (p^4 - p + 1)(p^2 + p + 1)$} & \multicolumn{2}{c}{$2p^4 + p^2 + 1$} \\
     \hline
     \end{tabular}
  \caption{Values of $\Sq[J]$ for $J = \{2,3\}, \{2\}, \{3\}$}
     \label{tab:ex-V}
 \end{table}

Note that the right side of \eqref{2des-examples} is $|O_J|/v$. This value, together with the boundary $\Bdy{J}$, is given for each $J$ in Table \ref{tab:V-orb}.

 \begin{table}[ht]
    \centering
    \begin{tabular}{llll}
    \hline
    $J$ & $\Bdy{J}$ & Pairs in $O_J$ & $|O_J|/v$ \\
    \hline
    $\{2,3\}$ & $\{1\}$ & same $\C_1$-class & $(e_1 - 1)=p^2+p$ \\
     $\{2\}$ & $\{3\}$ & same row, different columns & $(e_3 - 1)e_1=(p^4-p^2)(p^2+p+1)$ \\
    $\{3\}$ & $\{2\}$ & same column, different rows & $(e_2 - 1)e_1=(p^2-p)(p^2+p+1)$ \\
      \hline
    \end{tabular}
    \caption{Orbits of $G = G_1 \wr (G_2 \times G_3)$ on ordered pairs of distinct points}
    \label{tab:V-orb}
\end{table}

For each proper nonempty ancestral subset $J$, we compute the left side of \eqref{2des-examples} (recalling \eqref{eq:sum-squares}), and record the value in Table~\ref{tab:left}. We see that in each case the left side of \eqref{2des-examples} is equal to the right side, listed in the last column of Table \ref{tab:V-orb}. This concludes the proof that $\D = (\PP,B^G)$ is a $2$-design, by Theorem~\ref{t:gwp-design}.

 \begin{table}[ht]
    \centering
    \begin{tabular}{llll}
    \hline
    $J$ & $\Bdy{J}$ & $\sum_{S \subseteq \Bdy{J}} (-1)^{|S|} \Sq[J\cup S]$ & left side of \eqref{2des-examples} \\
    \hline
    $\{2,3\}$ & $\{1\}$ & $\Sq[\{2,3\}]-\Sq[I]$ & $p^2+p$ \\
    $\{2\}$ & $\{3\}$ & $\Sq[\{2\}]-\Sq[\{2,3\}]$ & $\left(p^4 - p^2\right)\left(p^2 + p + 1\right)$ \\
    $\{3\}$ & $\{2\}$ & $\Sq[\{3\}]-\Sq[\{2,3\}]$ & $(p^2-p)(p^2+p+1)$ \\   
 
    \hline
    \end{tabular}
    \caption{Values of $\sum_{S \subseteq \Bdy{J}} (-1)^{|S|} \Sq[J\cup S]$}
    \label{tab:left}
\end{table}
%
%
\end{example}

\subsubsection{Example for Figure \ref{fig:V-inv}}

Suppose that the point set $\PP$ admits a poset of partitions that corresponds to the inverted V-shaped poset $(I,\preccurlyeq)$ with nodes $1$, $2$, and $3$, as shown in Figure \ref{fig:V-inv}, which we show again below. 
    \begin{center}
    \includegraphics{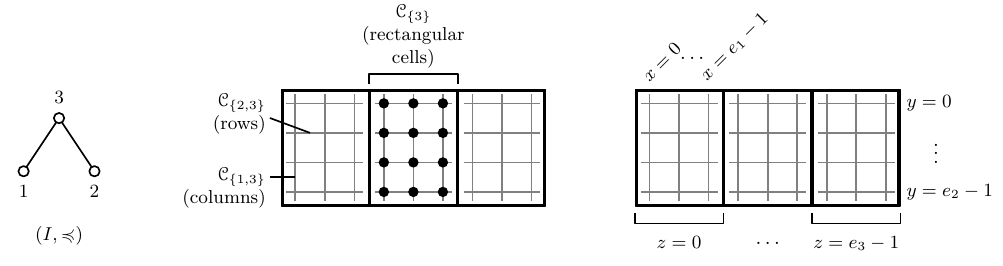}
    \end{center}
Let $G = (G_1 \times G_2) \wr G_3$, where each $G_i$ is a $2$-transitive subgroup of $S_{e_i}$. 
It follows from Theorems~\ref{thm:dir} and~\ref{thm:wr} that $G$ is a generalised wreath product group preserving this poset of partitions. In particular, these results imply that  the proper nonempty ancestral subsets of $I$ are $\{3\}$, $\{1,3\}$, and $\{2,3\}$, so the nontrivial point partitions of $\PP$ preserved by $G$ are:
\begin{itemize}
     \item $\C_{\{3\}}$ with classes $C_z=\{(x,y,z) \mid x\in\Delta_1, \ y\in\Delta_2\}$, for $z\in \Delta_3$; the $\C_{\{3\}}$-classes correspond to the rectangular cells in Figure~\ref{fig:V-inv};
      \item $\C_{\{1,3\}}$ with classes $C_{(x,z)}=\{(x,y,z) \mid y\in\Delta_2\}$ for $(x,z)\in \Delta_{\{1,3\}}=\Delta_1\times \Delta_3$; the $\C_{\{1,3\}}$-class $C_{(x,z)}$ represents a column in rectangular cell $C_z$ in Figure~\ref{fig:V-inv}.
     \item $\C_{\{2,3\}}$ with classes $C_{(y,z)}=\{(x,y,z) \mid x\in\Delta_1\}$ for $(y,z)\in \Delta_{\{2,3\}}=\Delta_2\times \Delta_3$; the $\C_{\{2,3\}}$-class $C_{(y,z)}$ represents a row in rectangular cell $C_z$ in Figure~\ref{fig:V-inv}.
\end{itemize} 

Thus for a point $\bm{\delta}=(x,y,z)\in\PP$, the third coordinate represents the rectangular cell $C_{z}\in\C_{\{3\}}$ containing $\bm{\delta}$, the second coordinate represents the row $C_{(y,z)}\in\C_{\{2,3\}}$ inside $C_{z}$ containing $\bm{\delta}$, and the first coordinate represents the column $C_{(x,z)}$ inside $C_{z}$ containing $\bm{\delta}$.

Note $ \CC^* = \{ \C_{A[i]} \ | \ i \in I \}$ where $A[1]=\{1,3\}=I$, $A[2]=\{2,3\}$, $A[3]=\{3\}$. 

\begin{example} \label{ex:V-inv}
Take $B = {\color{amber} B_1} \cup {\color{red} B_2} \cup {\color{blue} B_3}$ where ${\color{red} B_2} ={\color{red} B_{2,1} }\cup {\color{red}B_{2,2}}$ and
    \begin{align*}
    {\color{amber} B_1} &:= \{ (x,0,0) \ | \ 0 \leq x \leq p \}, \\
  {\color{red}  B_{2,1} }&:= \{ (p+1,\, y, \, 0) \ | \ 1 \leq y \leq p \}, \\
    {\color{red} B_{2,2}} &:= \{ (x, \, x - 1, \, 0) \ | \ p+2 \leq x \leq e_2\}, \ \text{and} \\
    {\color{blue} B_3 }&:= \{ (0,0,z) \ | \ 1\leq z\leq e_3-1\}.
    \end{align*}
Then $|{\color{amber} B_1}| = p+1$, $|{\color{red} B_2}| = p + (p^2 - 2p) = e_2 - 1$, and $|{\color{blue} B_3}| = e_3 - 1$, and thus
    \[ |B| = (p+1) + (e_2 - 1) + (e_3 - 1) = p^4 + 1 = q+1. \]
The points that comprise the block $B$ are shown in Figure \ref{fig:ex-V-inv}, as well as the subsets ${\color{amber} B_1}$, $ {\color{red} B_{2,1}}$, $ {\color{red} B_{2,2}}$, and ${\color{blue} B_3}$. Note that classes are numbered from left to right and top to bottom.

\begin{figure}[ht]
    \centering
    \includegraphics{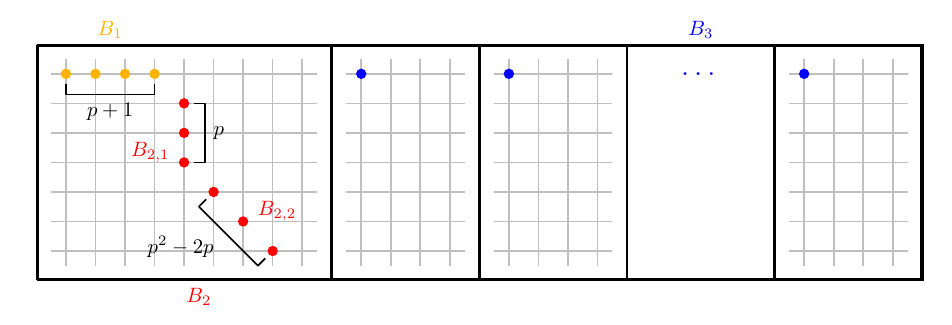}
    \caption{Distribution of points of $B$ in Example \ref{ex:V-inv}}
    \label{fig:ex-V-inv}
\end{figure}

In order to check that the structure $\D = (\PP,B^G)$ is a $2$-design, we will follow similar methods to  Example~\ref{ex:ch-grid}:
we need to check that condition \eqref{2des-examples} holds for any proper nonempty ancestral subset $J$ of $I$. As we mentioned above, the  proper nonempty ancestral subsets $J$ of $I$ are $\{1,3\}$, $\{2,3\}$ and $\{3\}$.
    Table \ref{tab:ex-V-inv} lists 
the relevant information to compute $\Sq[J]$ for each such $J$.

For instance, for $J = \{3\}$, the class $C_{0} \in \C_{\{3\}}$ contains all the points in ${\color{amber} B_1} \cup {\color{red} B_2}$, so that $(0)\chi_B = |B_1 \cup B_2| = (p+1) + (p^2 - p) = p^2 + 1$, while each of the remaining $p^4 - p^2$ classes in $\C_{\{3\}}$ contains one point in ${\color{blue} B_3}$, so that $(\bm{\nu})\chi_B = 1$ for each $\bm{\nu} \in \{1, \ldots, p^4 - p^2\}$. This gives the values in the first two rows under column $\{3\}$ of Table \ref{tab:ex-V-inv}, and thus
    \[
    \Sq[\{3\}]
    = 1 \cdot \big( (p + 1) + (p^2 - p) \big)^2 + (p^4 - p^2) \cdot 1^2
    = 2p^4 + p^2 + 1
    \]

\begin{table}[ht]
     \centering
     \begin{tabular}{r||r|r||r|r||r|r}
     \hline
     \multirow{2}{1.5cm}{\parbox[r]{1.5cm}{\raggedleft $J$ s.t. \\ $\bm{\nu} \in \Delta_J$}} & \multicolumn{2}{c||}{${\{1,3\}}$} & \multicolumn{2}{c||}{${\{2,3\}}$} & \multicolumn{2}{c}{${\{3\}}$} \\
     \cline{2-7}
     & $\# C_{\bm{\nu}}$ & $(\bm{\nu})\chi_B$ & $\# C_{\bm{\nu}}$ & $(\bm{\nu})\chi_B$ & $\# C_{\bm{\nu}}$ & $(\bm{\nu})\chi_B$ \\
     \hline\hline
     & $p^4 - p + 1$  & $1$  & $1$        & $p+1$ & $1$         & $p^2 + 1$ \\
     & $1$                  & $p$ & $p^4 - p$   & $1$   & $p^4 - p^2$ & $1$ \\
     \hline
     $\Sq[J]$\phantom{\Big|} & \multicolumn{2}{c||}{$p^4 + p^2 - p + 1$} & \multicolumn{2}{c||}{$p^4 + p^2 + p + 1$} & \multicolumn{2}{c}{$2p^4 + p^2 + 1$} \\
     \hline
     \end{tabular}
  \caption{Values of $\Sq[J]$ for $J = \{1,3\}, \{2,3\}, \{3\}$}
     \label{tab:ex-V-inv}
 \end{table}

Note that the right side of \eqref{2des-examples} is $|O_J|/v$ and this value together with the boundary $\Bdy{J}$ is given, for each $J$, in Table \ref{tab:Vinv-orb}.

 \begin{table}[ht]
    \centering
    \begin{tabular}{llll}
    \hline
    $J$ & $\Bdy{J}$ & Pairs in $O_J$ & $|O_J|/v$ \\
    \hline
    $\{1,3\}$ & $\{2\}$ & same $\C_3$-class, same column, different rows & $e_2 - 1=p^2-p$ \\
    $\{2,3\}$ & $\{1\}$ & same $\C_3$-class, same row, different columns & $e_1 - 1=p^2+p$ \\
    $\{3\}$ & $\{1,2\}$ & same $\C_3$-class, different rows and different columns & $(e_1 - 1)(e_2 - 1)=p^4-p^2$ \\
    \hline
    \end{tabular}
    \caption{Orbits of $G = (G_1 \times G_2) \wr G_3$ on ordered pairs of distinct points}
    \label{tab:Vinv-orb}
\end{table}

For each proper nonempty ancestral subset $J$, we compute the left side of \eqref{2des-examples} (recalling \eqref{eq:sum-squares}), and record the value in Table~\ref{tab:left-Vinv}. We see that in each case the left side of \eqref{2des-examples} is equal to the right side, listed in the last column of Table \ref{tab:Vinv-orb}.
This concludes the proof that $\D = (\PP,B^G)$ is a $2$-design, by Theorem~\ref{t:gwp-design}.

 \begin{table}[ht]
    \centering
    \begin{tabular}{llll}
    \hline
    $J$ & $\Bdy{J}$ & $\sum_{S \subseteq \Bdy{J}} (-1)^{|S|} \Sq[J\cup S]$ & left side of \eqref{2des-examples} \\
    \hline
    $\{1,3\}$ & $\{2\}$ & $\Sq[\{1,3\}]-\Sq[I]$ & $p^2-p$ \\
     $\{2,3\}$ & $\{1\}$ & $\Sq[\{2,3\}]-\Sq[I]$ & $p^2+p$ \\
    $\{3\}$ & $\{1,2\}$ & $\Sq[\{3\}]-\Sq[\{1,3\}]-\Sq[\{2,3\}]+\Sq[I]$ & $p^4-p^2$ \\
   
    \hline
    \end{tabular}
    \caption{Values of $\sum_{S \subseteq \Bdy{J}} (-1)^{|S|} \Sq[J\cup S]$}
    \label{tab:left-Vinv}
\end{table}
%
%
\end{example}

\subsection{The N-poset}

In this section, we give an infinite family of examples for just one poset $\I=(I,\preccurlyeq)$ with $|I|=4,$ namely for the \emph{N-poset}  in Figure~\ref{fig:posets-4}.  Thus the point set $\PP$ admits a poset of partitions corresponding to $\I$ with nodes $1$, $2$, $3$, and $4$, as shown in Figure \ref{fig:posets-4}, which we display again below. We have $\PP = \prod_{i=1}^4\Delta_i$ with each $|\Delta_i|=e_i\geq 2$, and we take the $e_i$ as follows, for some integer $p\geq2$: 
    \[ 
    e_1 = p^2 + p + 1, \ e_2 = p^2 - p + 1, \ e_3 = p^4 - p^2 + 1, \ e_4 = p^8 - p^4 + 1. 
    \]
For the $\I$-imprimitive group we take the  generalised wreath product  $G = \prod_{(I,\preccurlyeq)} (G_i,\Delta_i)$ corresponding to $\I$, where each $G_i$ is a $2$-transitive subgroup of $S_{e_i}$.  Note that  it is not possible to describe $G$ simply in terms of direct products and wreath products as was possible for the examples with $|I|=3$, since the hypotheses of neither Theorems~\ref{thm:dir} nor~\ref{thm:wr} are satisfied for $\I$.

We take $\Delta_i=\{0,1,2,\ldots, e_i-1\}\subset\mathbb{Z}$, and we note that $e_2<e_1<e_3<e_4$, and  
    \begin{equation}\label{v-forN}
       v = e_1e_2e_3e_4 = q^2 + q + 1, \quad \text{where} \ q = p^8.  
    \end{equation} 

\begin{center}
\includegraphics{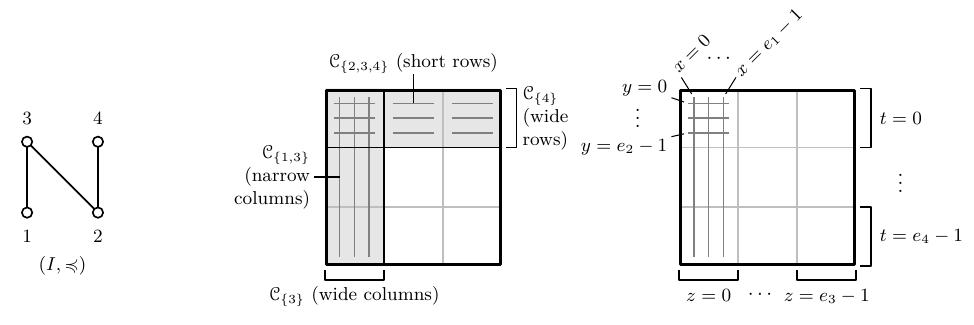}
\end{center}

As in the previous examples, each block has size $k = q + 1$, so that
    \[ \frac{k(k-1)}{v-1} = \frac{(q+1)q}{q^2 + q} = 1 \]
and for each ancestral subset $J$,  condition \eqref{2deswithmu} becomes condition \eqref{2des-examples}.

The proper nonempty ancestral subsets of $I$ are $\{3\}$, $\{4\}$, $\{3,4\}$, $\{1,3\}$, $\{1,3,4\}$ and $\{2,3,4\}$. These subsets together with $\varnothing$ and $I$ correspond to partitions in $(\CC,\preccurlyeq)$. On the other hand,
 \[
 \mbox{$A[1]=\{1,3\}, \ A[2]=\{2,3,4\}, \ A[3]=\{3\}, \ A[4]=\{4\}$},
 \]
so the partition-poset $\CC^* = \{ \C_{A[i]} \ | \ i \in I \}$. Classes in the partitions corresponding to the remaining ancestral sets $\{3,4\}$ and $\{1,3,4\}$ are  intersections of two classes from distinct partitions in $ \CC^*$, as explained in Section \ref{s:posetbkstr}. 
We now describe classes in each of the partitions in $(\CC,\preccurlyeq)$, noting that each of these partitions is preserved by $G$. 
 
\begin{itemize}
    \item $\C_{\{3\}}$ with classes $C_z = \{(x,y,z,t) \mid x\in\Delta_1, \ y\in\Delta_2, \ t\in\Delta_4\}$, for $z \in \Delta_3$; the $\C_{\{3\}}$-classes correspond to the wide columns in Figure~\ref{fig:posets-4}.
    \item $\C_{\{4\}}$ with classes $C_t = \{(x,y,z,t) \mid x\in\Delta_1, \ y\in\Delta_2, \ z\in\Delta_3\}$, for $t \in \Delta_4$; the $\C_{\{4\}}$-classes correspond to the wide rows in Figure~\ref{fig:posets-4}.
    \item $\C_{\{3,4\}}$ with classes $C_{(z,t)} = \{(x,y,z,t) \mid x\in\Delta_1, \  y\in\Delta_2\}$, for $(z,t)\in \Delta_3 \times \Delta_4$; the $\C_{\{3,4\}}$-classes correspond to cells that are intersections of wide rows and wide columns in Figure~\ref{fig:posets-4}.
    \item $\C_{\{1,3\}}$ with classes $C_{(x,z)} = \{(x,y,z,t) \mid y\in\Delta_2, \ t\in\Delta_4\}$ for $(x,z)\in \Delta_1 \times \Delta_3$; the $\C_{\{1,3\}}$-class $C_{(x,z)}$ represents a narrow column in wide column $C_z$ in Figure~\ref{fig:posets-4}.
    \item $\C_{\{1,3,4\}}$ with classes $C_{(x,z,t)} = \{(x,y,z,t) \mid y\in\Delta_2\}$ for $(x,z,t)\in \Delta_1\times \Delta_3\times\Delta_4$; the $\C_{\{1,3,4\}}$-classes correspond to intersections of $\C_{\{1,3\}}$-classes (narrow columns) with $\C_{\{4\}}$-classes (wide rows) in Figure~\ref{fig:posets-4}.
    \item $\C_{\{2,3,4\}}$ with classes $C_{(y,z,t)} = \{(x,y,z,t) \mid x\in\Delta_1\}$ for $(y,z,t)\in \Delta_2\times \Delta_3\times\Delta_4$; the $\C_{\{2,3,4\}}$-class $C_{(y,z,t)}$ corresponds to a short row inside class $C_{(z,t)}\in \C_{\{3,4\}}$, in Figure~\ref{fig:posets-4}.
\end{itemize}

Thus for a point $\bm{\delta} = (x,y,z,t)\in\PP$,  the fourth coordinate represents the wide row $C_{t}\in\C_{\{4\}}$ containing $\bm{\delta}$, the third coordinate represents the wide column $C_{z}\in\C_{\{3\}}$ containing $\bm{\delta}$, the first coordinate represents the narrow column $C_{(x,z)}\in\C_{\{1,3\}}$ inside the wide column $C_{z}$ containing $\bm{\delta}$, and the second coordinate represents short row $C_{(y,z,t)}$ inside cell $C_{(z,t)}$ containing $\bm{\delta}$. Now we define the base block $B$ with reference to Figure~\ref{fig:ex-N-block}.

\begin{example} \label{ex:N}
Take $B = \bigcup_{i=1}^4 B_i$, where ${\color{red} B_2} ={\color{red} B_{2,1}} \cup {\color{red}B_{2,2}}$, ${\color{green}B_4} = {\color{green}B_{4,1}} \cup {\color{green}B_{4,2}} $ and
    \begin{align*}
    {\color{amber} B_1} &:= \{ (x,\, 0,\, 0,\, 0) \ | \ 0 \leq x \leq p \}, \\ 
   {\color{red} B_{2,1}} &:= \{ (p+1,\, y,\, 0,\, 0) \ | \ 1 \leq y \leq p \}, \\
    {\color{red} B_{2,2}} &:= \{ (x,\, x - 1,\, 0,\, 0) \ | \ p+2 \leq x \leq e_2 \}, \\
    {\color{blue} B_3} &:= \left\{ (0,\, 0,\, z,\, 0) \ | \ 1\leq z\leq e_3-1 \right\}, \\
    {\color{green}B_{4,1}} &:= \left\{ (x,\, 0,\, z,\, (z - 1)\left(p^4 + p^2\right)+t) \ | \ 0 \leq x \leq 2p-1, \ 1\leq z\leq e_3-1  , \right. \\[-3pt]
        &\phantom{:=}\;\; \left. x p^2 + 1 \leq t \leq (x + 1) p^2 \right\}, \ \text{and} \\
    {\color{green} B_{4,2}} &:= \left\{ (x,\, 0,\, z,\, (z - 1)\left(p^4 + p^2\right)+t) \ | \ 2p \leq x \leq p^2 + p - 1, \ 1 \leq z \leq e_3-1, \right. \\[-3pt]
        &\phantom{:=}\;\; \left.2p^3 + (x - 2p)(p^2 - p) + 1 \leq t \leq 2p^3 + (x - 2p + 1)(p^2 - p) \right\}.\\ 
     \end{align*}
    The points that comprise the block $B$ are shown in Figure \ref{fig:ex-N-block}, as well as the subsets ${\color{amber} B_1}$, ${\color{red} B_{2,1}}$, ${\color{red} B_{2,2}}$,  ${\color{blue} B_3}$ and ${\color{green}B_4}$. Note that classes are numbered from left to right and top to bottom. 
    Observe that $ {\color{amber} B_1},  {\color{red} B_2}, {\color{blue} B_3}$ are all contained in $C_0\in\C_{\{4\}}$ (the wide row at the top of Figure~\ref{fig:posets-4}), and are the same subsets as those in Example~\ref{ex:V-inv}. The `zoom in' on the right in Figure~\ref{fig:ex-N-block} represents the intersection of ${\color{green}B_4}$ with each wide column $C_z$, which we now explain in more detail.
    
To understand ${\color{green}B_4}$, we consider a fixed value of $z\in\Delta_3\setminus\{0\}$, that is, a fixed wide column $C_z$. The fourth coordinates of the  points of  ${\color{green}B_{4}} $  with third coordinate $z$ are all of the form $(z - 1)\left(p^4 + p^2\right)+t$;  there is one point in $ {\color{green}B_{4,1}} $ for each value of $t$ from $1$ to $2p^3$; and  there is one point in $ {\color{green}B_{4,2}} $ for each value of $t$ from $2p^3+1$ to $p^4+p^2$.
Recall that distinct fourth coordinates correspond to distinct wide rows. The second coordinate of all these points is $0$ indicating that the points are all in the top short row of the cell containing them; and the first coordinates  of all these points indicate which narrow column within $C_z$ the points belong to.
 
Thus from top to bottom, and starting at wide row $C_{(z - 1)\left(p^4 + p^2\right)+1}$, we have  one point in narrow column $C_{(0,z)}$ of the first $p^2$ wide rows, then one point in narrow column $C_{(1,z)}$ of the next $p^2$ wide rows, and so on, up to one point in narrow  column $C_{(2p-1,\,z)}$ of $p^2$ wide rows, finishing in the $2p^3$-th wide row. 
Then for the next narrow column  $C_{(2p,z)}$,  we have  one point from each of the next $p^2-p$ wide rows, continuing to the last column $C_{(p^2+p-1,\,z)}$, where we have one point from each of the next  $p^2-p$ wide rows finishing at the $(p^4+p^2)$-th wide row. The last wide row containing a point of ${\color{green}B_4}$ corresponds to $z=e_3-1=p^4-p^2$ and is $C_u\in\C_{\{4\}}$ where  $u=(z - 1)\left(p^4 + p^2\right)+p^4+p^2=(p^4-p^2 - 1)\left(p^4 + p^2\right)+p^4+p^2=p^8-p^4=e_4-1$, that is to say, $C_u$ is in fact the last wide row.
    
Now $|{\color{amber} B_1}| = p + 1$, $|{\color{red} B_2}| = p + (e_2 - (p+1)) = e_2 - 1= p^2 - p$, $|{\color{blue} B_3}| = e_3 - 1=p^4-p^2$, and
    \begin{align*}
    |{\color{green}B_4}|
    &= (e_3 - 1)\left[ 2p\cdot p^2 + (p^2-p)^2\right]  \\
    &= (p^4-p^2)(p^4+p^2)\\
    &= p^8-p^4= e_4 - 1,
    \end{align*}
so that
    \[ |B| = (p+1)+(p^2-p)+(p^4-p^2)+(p^8-p^4) 
    = p^8+1=q + 1. 
    \]

\begin{figure}[ht]
    \centering
    \includegraphics[width=\textwidth]{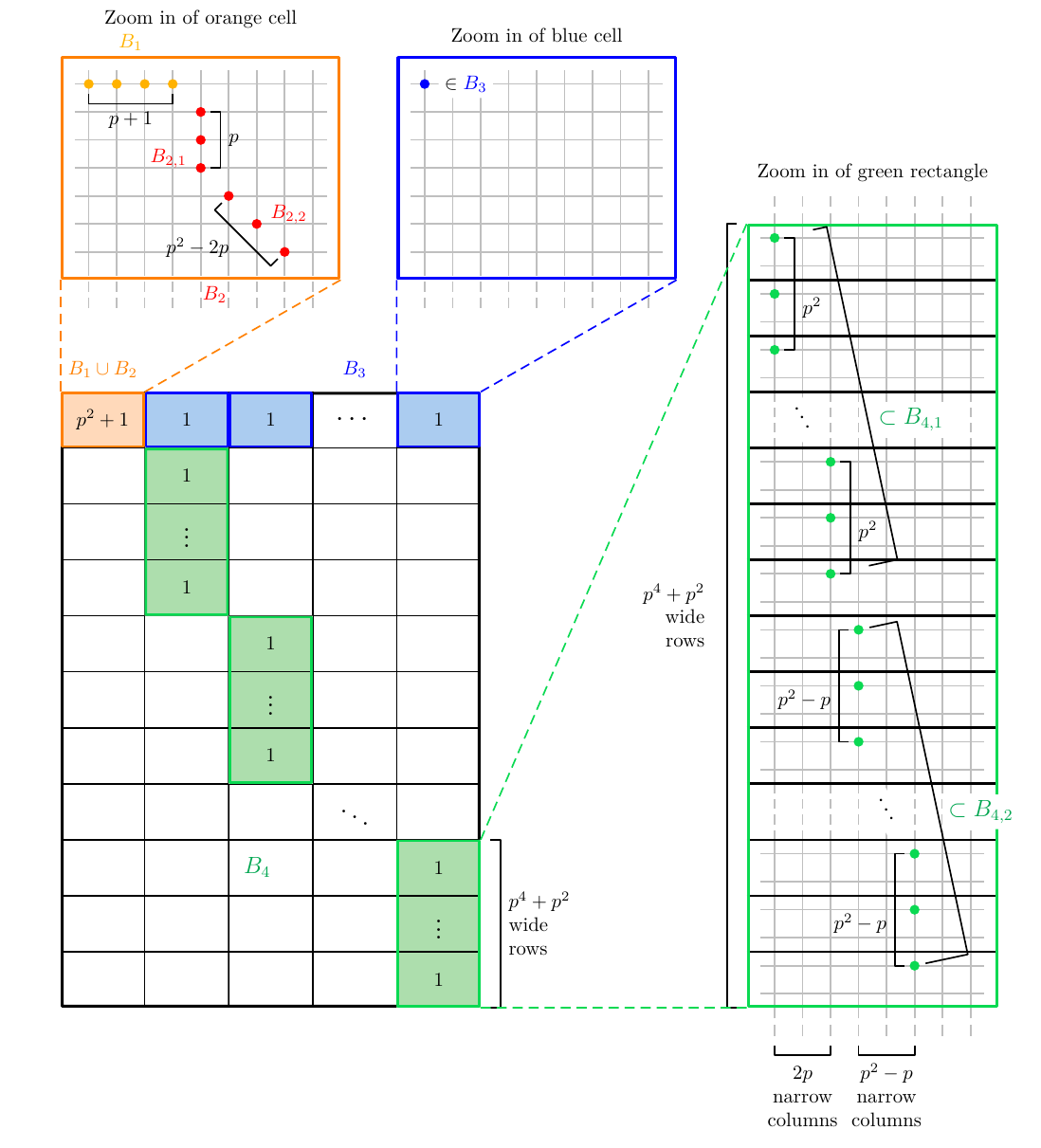}
    \caption{Figure for Example \ref{ex:N}}
    \label{fig:ex-N-block}
\end{figure}

In order to check that the incidence structure $\D = (\PP,B^G)$ is a $2$-design, we follow similar methods to  Example~\ref{ex:ch-grid}:
we need to check that condition \eqref{2des-examples} holds for each proper nonempty ancestral subset $J$ of $I$, which we recall are   $\{3\}$, $\{4\}$, $\{3,4\}$, $\{1,3\}$, $\{1,3,4\}$ and $\{2,3,4\}$.
Recalling \eqref{eq:sum-squares}, Table \ref{tab:ex-N} lists the relevant information needed to compute $\Sq[J]$ for each of these subsets $J$.    

\begin{table}[ht]
     \centering
     \begin{tabular}{r||r|r||r|r||r|r}
     \hline
     \multirow{2}{1.5cm}{\parbox[r]{1.5cm}{\raggedleft $J$ s.t. \\ $\bm{\nu} \in \Delta_J$}} & \multicolumn{2}{c||}{${\{1,3,4\}}$} & \multicolumn{2}{c||}{${\{2,3,4\}}$} & \multicolumn{2}{c}{${\{1,3\}}$} \\
     \cline{2-7}
     & $\# C_{\bm{\nu}}$ & $(\bm{\nu})\chi_B$ & $\# C_{\bm{\nu}}$ & $(\bm{\nu})\chi_B$ & $\# C_{\bm{\nu}}$ & $(\bm{\nu})\chi_B$ \\
     \hline\hline
     & $1$           & $p$  & $1$       & $p+1$ & $1$                    & $p$ \\
     & $p^8 - p + 1$ & $1$& $p^8 - p$ & $1$   & $p^2 - p + 1$          & $1$ \\
     &               &     &           &       & $p^4 - p^2$            & $p^2 + 1$ \\
     &               &     &           &       & $(2p - 1)(p^4 - p^2)$  & $p^2$ \\
     &               &     &           &       & $(p^2 - p)(p^4 - p^2)$ & $p^2 - p$ \\
     \hline
     $\Sq[J]$\phantom{\Big|} & \multicolumn{2}{c||}{$p^8 + p^2 - p + 1$} & \multicolumn{2}{c||}{$p^8 + p^2 + p + 1$} & \multicolumn{2}{c}{\parbox{4.5cm}{\centering $e_4(e_2 + 1) + p^4-1$}} \\
     \hline
     \hline
     \multirow{2}{1.5cm}{\parbox[r]{1.5cm}{\raggedleft $J$ s.t. \\ $\bm{\nu} \in \Delta_J$}} & \multicolumn{2}{c||}{${\{3,4\}}$} & \multicolumn{2}{c||}{${\{3\}}$} & \multicolumn{2}{c}{${\{4\}}$} \\
     \cline{2-7}
     & $\# C_{\bm{\nu}}$ & $(\bm{\nu})\chi_B$ & $\# C_{\bm{\nu}}$ & $(\bm{\nu})\chi_B$ & $\# C_{\bm{\nu}}$ & $(\bm{\nu})\chi_B$ \\
     \hline\hline
     & $1$         & $p^2 + 1$ & $1$         & $p^2 + 1$       & $1$         & $p^4 + 1$ \\
     & $p^8 - p^2$ & $1$       & $p^4 - p^2$ & $p^4 + p^2 + 1$ & $p^8 - p^4$ & $1$ \\
     \hline
     $\Sq[J]$\phantom{\Big|} & \multicolumn{2}{c||}{$p^8 + p^4 + p^2 + 1$} & \multicolumn{2}{c||}{\parbox{3.25cm}{\centering $e_1e_2(e_4 - 1) + p^8 + p^4 + p^2 + 1$}} & \multicolumn{2}{c}{$2p^8 + p^4 + 1$} \\
     \hline
     \end{tabular}
  \caption{Values of $\Sq[J]$ for ancestral $J$}
     \label{tab:ex-N}
 \end{table}

We give details for the case $J = \{1,3\}$, which is the most complicated case. Recall that $\C_{\{1,3\}}$-classes are narrow columns inside a wide column.
Inside wide column $C_0$ we have one $\C_{\{1,3\}}$-class $C_{(p+1,\,0)}$ containing $p$ points of $B$ forming the set ${\color{red} B_{2,1}}$, and an additional $p^2-p+1$ points coming from pairwise distinct $\C_{\{1,3\}}$-classes $C_{(x,0)}$, namely $p+1$ points in the classes with $0\leq x \leq p$ forming ${\color{amber} B_1}$, and $p^2-2p$ points in the classes with $p+2\leq x\leq e_2$  forming ${\color{red} B_{2,2}}$. This explains the first two rows for case $J = \{1,3\}$ in Table \ref{tab:ex-N}.
Now consider an arbitrary wide column $C_z$ with $z\neq 0$. (There are $e_3-1=p^4-p^2$ of these.) The narrow column $C_{(0,z)}$ in $C_z$ contains one point in ${\color{blue} B_3}$ and $p^2$ points in ${\color{green}B_{4,1}}$, corresponding to the third  row for case $J = \{1,3\}$ in Table \ref{tab:ex-N}. Among the narrow columns $C_{(x,z)}$ in $C_z$ with $x\ne0$: for $1\leq x\leq 2p-1$, $C_{(x,z)}$ contains $p^2$
points of ${\color{green}B_{4,1}}$, corresponding to the fourth row for case $J = \{1,3\}$ in Table \ref{tab:ex-N}. Finally, for $2p\leq x\leq p^2+p-1$, $C_{(x,z)}$ contains $p^2-p$
points of ${\color{green}B_{4,2}}$, corresponding to the fifth row for case $J = \{1,3\}$ in Table \ref{tab:ex-N}. To compute $\Sq[\{1,3\}]$ from \eqref{eq:sum-squares} using the information in Table~\ref{tab:ex-N} obtained so far, we proceed as follows (and obtain the entry for $\Sq[\{1,3\}]$ in Table~\ref{tab:ex-N}):

\begin{align*}
\Sq[\{1,3\}]
&= 1 \cdot p^2 + (p^2 - p + 1) \cdot 1^2 + (p^4 - p^2)\big( (p^2 + 1)^2 + (2p-1)(p^2)^2 + (p^2 - p)(p^2 - p)^2\big) \\
&= p^2 + e_2 + (p^4 - p^2) \big( p^4 + 2p^2 + 1 + 2p^5 - p^4 + p^6 - 3p^5 + 3p^4 - p^3 \big) \\
&= p^2 + e_2 + (p^4 - p^2) \big( p^6 - p^5 + 3p^4 - p^3 + 2p^2 + 1 \big) \\
&= p^2 + e_2 + (p^4 - p^2) \big( (p^4 + p^2)(p^2 - p + 2) + 1 \big) \\
&= p^2 + e_2 + (p^8 - p^4)(p^2 - p + 2) + (p^4 - p^2) \\
&= e_2 + (e_4 - 1)(e_2 + 1) + p^4\\
&= e_4(e_2 + 1) + p^4-1
\end{align*}


Note that the right side of \eqref{2des-examples} is $|O_J|/v$ and this value together with the boundary $\Bdy{J}$ is given, for each non-empty ancestral subset  $J$, in Table \ref{tab:N-orb}.

\begin{table}[ht]
    \centering
    \begin{tabular}{lll}
    \hline
    $J$ & $\Bdy{J}$ & $|O_J|/v$ \\
    \hline\hline
    $\{1,3,4\}$ & $\{2\}$ & $e_2 - 1=p^2-p$  \\
    $\{2,3,4\}$ & $\{1\}$ & $e_1 - 1=p^2+p$ \\
    $\{1,3\}$ & $\{4\}$ & $(e_4 - 1)e_2=(p^8-p^4)(p^2-p+1)$ \\
    $\{3,4\}$ & $\{1,2\}$ & $(e_1 - 1)(e_2 - 1)=p^4-p^2$ \\
    $\{3\}$ & $\{1,4\}$ & $(e_1 - 1)(e_4 - 1)e_2=(p^2+p)(p^8-p^4)(p^2-p+1)$ \\
    $\{4\}$ & $\{3\}$ & $(e_3 - 1)e_1e_2=p^8-p^2$ \\
    \hline
    \end{tabular}
    \caption{Orbit sizes of $G$ on ordered pairs of distinct points}
    \label{tab:N-orb}
\end{table}

For each proper nonempty ancestral subset $J$, we compute the left side of \eqref{2des-examples}, and record the value in Table~\ref{tab:left-N}. Note that this uses $\Sq[I]=k=p^8+1$ (see Lemma~\ref{lem:array-sqs}). We see that in each case the left side of \eqref{2des-examples} is equal to the right side, listed in the last column of Table \ref{tab:N-orb}. Note that for $J=\{3,4\}$ and $J=\{3\}$ we may use results from other cases to shorten the computation. Thus equation \eqref{2des-examples} holds for each proper nonempty ancestral subset $J$, and hence  $\D = (\PP,B^G)$ is a $2$-design, by Theorem~\ref{t:gwp-design}.

 \begin{table}[ht]
    \centering
    \begin{tabular}{llll}
    \hline
    $J$ & $\Bdy{J}$ & $\sum_{S \subseteq \Bdy{J}} (-1)^{|S|} \Sq[J\cup S]$ & left side of \eqref{2des-examples} \\
    \hline
    $\{1,3,4\}$ & $\{2\}$ & $\Sq[\{1,3,4\}]-\Sq[I]$ & $p^2-p$ \\
     $\{2,3,4\}$ & $\{1\}$  & $\Sq[\{2,3,4\}]-\Sq[I]$ & $p^2+p$ \\
    $\{1,3\}$ & $\{4\}$ & $\Sq[\{1,3\}]-\Sq[\{1,3,4\}]$ & $(e_4 - 1)e_2$ \\
     $\{3,4\}$ & $\{1,2\}$ &$\Sq[\{3,4\}]-\Sq[\{1,3,4\}]-\Sq[\{2,3,4\}]+\Sq[I]$&$p^4-p^2$ \\
    $\{3\}$ & $\{1,2\}$ &$\Sq[\{3\}]-\Sq[\{1,3\}]-\Sq[\{3,4\}]+\Sq[\{1,3,4\}]$ &  $(e_1 - 1)(e_4 - 1)e_2$\\
    $\{4\}$ & $\{3\}$ &$\Sq[\{4\}]-\Sq[\{3,4\}]$ & $p^8-p^2$\\  
    \hline
    \end{tabular}
    \caption{Values of $\sum_{S \subseteq \Bdy{J}} (-1)^{|S|} \Sq[J\cup S]$}
    \label{tab:left-N}
\end{table}

\end{example}

\section{Declarations}
All three authors wrote and reviewed the main manuscript text, and C.A. prepared all figures.

{\bf Conflict of interest:} The authors have no competing interests to declare that are relevant to the content of this article.

{}

\end{document}